\documentclass[11pt, reqno]{amsart}

\usepackage[left=3cm, right=3cm, top=2.5cm, bottom=2.5cm]{geometry}

\usepackage[utf8]{inputenc}

\usepackage{amssymb, amsmath, amsthm, esint}
\usepackage{textcomp}
\allowdisplaybreaks
\usepackage{latexsym}
\usepackage{graphicx}

\usepackage{hyperref}
\usepackage{enumerate}

\newtheorem{thm}{Theorem}[section]
\newtheorem{prop}[thm]{Proposition}
\newtheorem{lemma}[thm]{Lemma}
\newtheorem{cor}[thm]{Corollary}

\theoremstyle{definition}

\newtheorem{defi}[thm]{Definition}
\newtheorem{ex}[thm]{Example}
\newtheorem{rem}[thm]{Remark}

\newcommand{\R}{\mathbb R}
\newcommand{\N}{\mathbb N}
\newcommand{\bb}{\mathbb{B}}

\newcommand{\bD}{\boldsymbol{D}}
\newcommand{\nodal}{\mathcal{N}}
\newcommand{\I}{{\rm I}}
\newcommand{\II}{{\rm I\!I}}
\newcommand{\III}{{\rm I\!I\!I}}
\newcommand{\IV}{\ensuremath{\widetilde{\rm I\!I\!I}}}
\newcommand{\vtheta}{\theta}
\newcommand{\essinf}{\operatorname*{ess\,inf}}
\DeclareMathOperator{\per}{Per}
\newcommand{\bOm}{\overline\Omega}

\DeclareMathOperator{\loc}{loc}
\DeclareMathOperator{\Lip}{Lip}

\DeclareMathOperator{\supp}{supp}
\DeclareMathOperator{\cH}{\mathcal{H}}
\DeclareMathOperator{\grad}{grad}
\DeclareMathOperator{\dive}{div}

\newcommand\eps{\ensuremath{\varepsilon}}

\title{Nodal counts for the Robin problem on Lipschitz domains}
\author{Katie Gittins}
\address{Department of Mathematical Sciences, Durham University,
Mathematical Sciences \& Computer Science Building,
Upper Mountjoy Campus, Stockton Road,
Durham, DH1 3LE,
United Kingdom.}
\email{katie.gittins@durham.ac.uk}

\author{Asma Hassannezhad}
\address{University of Bristol,
School of Mathematics,
Fry Building,
Woodland Road,
Bristol, 
BS8 1UG, United Kingdom}
\email{asma.hassannezhad@bristol.ac.uk}

\author{Corentin Léna}
\address{Università degli Studi di Padova, Dipartimento di Tecnica e Gestione dei Sistemi Industriali (DTG), Stradella S. Nicola 3, 36100 Vicenza, Italy}
\address{Università degli Studi di Padova, Dipartimento di Matematica ``Tullio Levi-Civita'', via Trieste 63, 35121 Padova, Italy}
\email{corentin.lena@unipd.it}

\author{David Sher}
\address{DePaul University, Department of Mathematical Sciences, 2320 N. Kenmore Ave., Chicago IL 60614, USA}
\email{dsher@depaul.edu}


\date{\today}

\begin{document}

\begin{abstract}
We consider the Courant-sharp eigenvalues of the Robin Laplacian for bounded, connected, open sets in $\R^n$, $n \geq 2$, with Lipschitz boundary. We prove Pleijel's theorem which implies that there are only finitely many Courant-sharp eigenvalues in this setting as well as an improved version of Pleijel's theorem, extending previously known results that required more regularity of the boundary. In addition, we obtain an upper bound for the number of Courant-sharp Robin eigenvalues of a bounded, connected, convex, open set in $\R^n$ with $C^2$ boundary that is explicit in terms of the geometric quantities of the set and the norm sup of the negative part of the Robin parameter. 
\end{abstract}

\maketitle

\section{Introduction}
Let $\Omega \subset \R^n$ be a bounded, connected, open set with Lipschitz boundary. Let $h \in L^{\infty}(\partial\Omega, \R)$.
We consider the eigenvalue problem for the Robin Laplacian
\begin{equation}\label{eq:roblap}
    \begin{cases}
        \Delta u = \mu u & \text{ in } \Omega,\\
        \partial_\nu u + h u = 0 &\text{ on } \partial \Omega, 
    \end{cases}
\end{equation}
where $\Delta = - \dive \grad$ is the positive Laplacian and $\nu$ is the unit outward-pointing normal along $\partial \Omega$. The eigenvalues of the Robin Laplacian on $L^2(\Omega)$, counted with multiplicity, can be written as
$$ \mu_1(\Omega, h) \leq \mu_2(\Omega, h) \leq \dots \leq \mu_k(\Omega, h) \leq \dots \nearrow +\infty.$$
We use the notation $\Delta_\Omega^{R,h}$ to refer to the Robin Laplacian in \eqref{eq:roblap}. The case where $h \equiv 0$ corresponds to the Neumann Laplacian and the case where $h \to +\infty$ corresponds to the Dirichlet Laplacian.

Let $u_k$ be an eigenfunction of $\Delta_\Omega^{R,h}$ corresponding to the $k$-th eigenvalue $\mu_k(\Omega, h)$. The connected components of $\Omega \setminus \overline{\{x \in \Omega : u_k(x) = 0\}}$ are called nodal domains of $u_k$.
Courant's Nodal Domain theorem asserts that any eigenfunction $u_k$ corresponding to the $k$-th eigenvalue $\mu_k(\Omega, h)$ has at most $k$ nodal domains. In the case where $\mu_k(\Omega, h)$ has an eigenfunction with exactly $k$ nodal domains, we say that $\mu_k(\Omega, h)$ is a Courant-sharp eigenvalue and $u_k$ is a Courant-sharp eigenfunction.

Let $\mathcal{N}_\Omega^h(k)$ denote the number of nodal domains of the eigenfunction of $\Delta_\Omega^{R,h}$ corresponding to the $k$-th eigenvalue (counted with multiplicities).
Pleijel's theorem \cite{Pl56,BM82} asserts that for a bounded, connected, open set $\Omega \subset \R^n$, 
\begin{equation}\label{eq:pldir}
\limsup_{k\to\infty}\frac{\nodal_\Omega^{+\infty}(k)}{k}\le \gamma(n),
\end{equation}
where $\gamma(n):=(2\pi)^n/\omega_n^2 j^n_{\frac{n-2}{2}} <1,$ $\omega_n$ is the Lebesgue measure of a ball in $\R^n$ of radius $1$ and $j_{\frac{n-1}{2}}$ is the smallest positive zero of the Bessel function $J_{\frac{n-2}{2}}$. Pleijel's theorem for the Dirichlet Laplacian also holds in the Riemannian setting, see \cite{Peet57,BM82}, and results have recently been obtained in the sub-Riemannian setting, see \cite{FH24}.

It has been shown that the upper bound in \eqref{eq:pldir} is not sharp. Improved versions of Pleijel's theorem that obtain a smaller constant (depending only on $n$) in the right-hand side of \eqref{eq:pldir} have been obtained in \cite{D14, S14, B15}.
Pleijel's theorem implies that there are only finitely many Courant-sharp Dirichlet eigenvalues of $\Omega$.
This result, called the weak Pleijel theorem, holds for any open set $\Omega \subset \R^n$ of finite Lebesgue measure (see, e.g., \cite{vdBkG16cs}). 

The extension of Pleijel's theorem to the Neumann and Robin eigenvalue problems, as well as investigating geometric bounds on  Courant-sharp eigenvalues and the number of them, has been the subject of several studies summarised below over the last few years. In this paper, we further these developments by focusing on the Robin eigenvalue problem on a Lipschitz domain $\Omega$ with an arbitrary Robin parameter $h\in L^\infty(\partial\Omega)$.\\

De Ponti, Farinelli, and Violo \cite{DFV} proved Pleijel's theorem in the setting of metric-measure spaces, in particular, for the Neumann problem on a so-called uniform domain in the non-smooth setting of RCD spaces. Since Lipschitz domains are examples of uniform domains in $\R^n$ , see e.g. \cite{Jones,MS79,vai88}, their results imply Pleijel's theorem for the Neumann problem on Lipschitz domains in $\mathbb{R}^n$. Pleijel's theorem for the Neumann problem was first proved for compact surfaces with piecewise real-analytic boundary \cite{P} and then extended to bounded domains in $\R^n$ with $C^{1,1}$ boundary \cite{cL16}. In \cite{BCM}, Pleijel's theorem was shown to hold for chain domains (roughly speaking, a collection of bounded, disjoint, planar domains and a collection of thin necks joining these domains) and it was shown that the Courant-sharp Neumann eigenvalues are bounded uniformly in terms of the geometric quantities of a family of such domains.
Pleijel's theorem for the Robin problem on $C^{1,1}$ domains in $\R^n$ in the case where $h \geq 0$ was proven in \cite{cL16}, while in \cite{HS} it was shown that an improved version of Pleijel's theorem for the Robin problem on $C^{1,1}$ domains holds for any $h \in L^\infty(\partial\Omega)$. 

We extend the result in \cite{DFV} for the Neumann problem on Lipschitz domains to the Robin problem on Lipschitz domains for any $h \in L^\infty(\Omega)$ (see Theorem \ref{thm:robinplej}), {relaxing the boundary regularity required in \cite{HS}.}

The proof of Pleijel's theorem for the Dirichlet Laplacian relies on two main ingredients: the Faber-Krahn inequality and Weyl's law. The primary challenge in extending this proof to other boundary conditions lies in adapting the argument involving the Faber-Krahn inequality.
A key step in the proof is to obtain an upper bound for the {Neumann} Rayleigh quotient of a Robin eigenfunction on a nodal domain in terms of the corresponding eigenvalue, the geometry of the underlying domain and the Robin parameter. To obtain such a bound, we make use of the fact that any Lipschitz domain has an outward-pointing vector field (see Section \ref{ss:rq}).
Moreover,  in this setting, we improve the upper bound in Pleijel's theorem   (see Theorem \ref{thm:pleijel improved}). The proof of an improved version of Pleijel's theorem for the Robin problem on a Lipschitz domain is rather intricate and requires first establishing a quantitative version of certain results proved in \cite{DFV}, including a quantitative form of the Faber-Krahn inequality for mixed Dirichlet-Neumann eigenvalues on domains with `small' volume (see Proposition \ref{prop:fkq}).  \\   

Although in the paper we only consider Lipschitz domains in $\mathbb{R}^n$, one can show that Pleijel's theorem and its improved version for the Robin problem also hold in the Riemannian setting. This is because the results of \cite{DFV} are valid in a very general context, and other techniques used in the proof can be adapted to the Riemannian setting, as discussed in \cite[Theorem 2.5]{HS} and \cite{D14,BM82}. \\ 

With Pleijel's theorem in hand, it is natural to investigate how many Courant-sharp eigenvalues there are and how the geometry of the underlying domain and the Robin parameter $h$ can be used to quantify this number. Roughly speaking, while Weyl's law plays a key role in the final step of the proof of Pleijel's theorem, the remainder in Weyl's law plays an important role in getting more information on the count of the Courant-sharp eigenvalues. Hence, the bound is sensitive to the geometry of the domain and its boundary. Employing bounds on the remainder has been successfully used in \cite{vdBkG16cs,BH16}.  In \cite{BH16}, an upper bound for the number of Courant-sharp Dirichlet eigenvalues of a bounded, open set in $\R^2$ with $C^2$ boundary is obtained in terms of the area, the perimeter, bounds on the principal curvatures and the cut-distance to the boundary.
In \cite{vdBkG16cs}, several upper bounds for the number of Courant-sharp Dirichlet eigenvalues of various Euclidean domains are obtained. In particular, an upper bound for the number of Courant-sharp Dirichlet eigenvalues of a bounded, open, convex set in $\R^n$ is obtained in terms of the $(n-1)$-dimensional Hausdorff measure of the boundary and the volume of the set. To extend the results for the Dirichlet Laplacian to the Neumann and Robin problems, more regularity and geometric assumptions are needed. For the Neumann problem and the Robin problem with $h \geq 0$, upper bounds for the number of Courant-sharp eigenvalues of a bounded, open, convex set in $\R^n$, $n \geq 2$, with $C^2$ boundary are obtained in \cite{GL18} in terms of the volume of the set, the isoperimetric ratio and the principal curvatures of the boundary. 

We obtain an upper bound for the number of Courant-sharp Robin eigenvalues of a bounded, open, connected, convex set $\Omega \subset \R^n$, $ n \geq 2$, with $C^2$ boundary, that is explicit in terms of the geometric quantities of $\Omega$ and the Robin parameter (see Theorem~\ref{thm:number}). To do this, we first obtain an inequality comparing the Robin counting function to a shifted Neumann counting function and employ a result from \cite[Appendix A]{GL18} which gives an upper bound on the Neumann counting function for such domains. We then obtain an upper bound for the largest Courant-sharp Robin eigenvalue  (see Theorem~\ref{thm:eig}), for which the convexity of the domain is not required. Our result then follows by substituting the bound for the eigenvalue into the bound for the counting function.

\section*{Plan of the paper}
In Section \ref{ss:rq}, we obtain an upper bound for the Neumann Rayleigh quotient of a Robin eigenfunction on a nodal domain. We then employ this bound together with techniques from \cite{DFV} to prove Pleijel's theorem in Section \ref{ss:pleij}, where we also prove an improved version of Plejiel's theorem. However, the proof of the latter is elaborate and entails additional technical steps. For the convenience of the reader and to illustrate the additional steps required to obtain an improved version, we present the proof of Pleijel's theorem first, followed by the improved version. In order to obtain the improved version of Pleijel's theorem, we {in particular} require a quantitative version of the Faber-Krahn inequality for domains with small volume in this setting, which we prove in the Appendix. In Section \ref{sec:uppernumber}, we obtain an explicit upper bound on the number of Courant-sharp Robin eigenvalues of a bounded, open, connected, convex set $\Omega \subset \R^n$, $ n \geq 2$, with $C^2$ boundary.

\section*{Acknowledgement}
A.\,H. acknowledges  support of EPSRC grant EP/T030577/1. C.\,L. acknowledges support from the INdAM GNAMPA Project  ‘‘\emph{Operatori differenziali e integrali in geometria spettrale}’’ (CUP E53C22001930001). D.\,S. acknowledges support from the AMS-Simons grant 501949-9208.
The authors are very grateful to the referee for many helpful suggestions and comments.

\newpage
\section{Neumann Rayleigh quotients of Robin eigenfunctions on nodal domains}\label{ss:rq}

Throughout the paper we assume that $\Omega\subset \mathbb{R}^n$ is a Lipschitz domain, that is, an open, bounded and connected Lipschitz set.

\begin{defi} \label{defOutward} Let $\Omega$ be a Lipschitz domain. We call the mapping $F: \mathbb{R}^n \to \mathbb{R}^n$ an \emph{outward-pointing vector field} relative to $\Omega$ if
    \begin{enumerate}[(i)]
        \item $F$ is of class $C^\infty$ with compact support;
        \item there exists $\gamma_F(\Omega)>0$ such that, for a.e. $x\in\partial \Omega$, we have
            \[F(x)\cdot \nu(x)\ge \gamma_F(\Omega).\]
    \end{enumerate}
\end{defi}

The following proposition is due to Mitrea-Taylor \cite[Appendix A]{MT99} and to Verchota~\cite{Ver82}. For the convenience of the reader, we provide a proof.

\begin{prop}\label{propOutward} Any Lipschitz domain $\Omega$ has an outward-pointing vector field.
\end{prop}

\begin{proof}For any $x_0\in \partial \Omega$, we call $F_{x_0}:\mathbb R^n\to\mathbb R^n$ a \emph{local outward-pointing vector field} at $x_0$, relative to $\Omega$, if
\begin{enumerate}[(i)]
    \item $F_{x_0}$ is of class $C^\infty$ with compact support;
    \item for all $x\in \partial \Omega$ such that $\nu(x)$ exists, $F_{x_0}(x)\cdot \nu(x)\ge 0$;
    \item there exists $\gamma_{x_0}>0$ and an open neighborhood $U_{x_0}$ of $x_0$ such that, for all $x\in U_{x_0}\cap \partial\Omega$ at which $\nu(x)$ exists,
    \[F_{x_0}(x)\cdot \nu(x)\ge \gamma_{x_0}.\]
\end{enumerate}

If, relative to $\Omega$, there exists a local outward-pointing vector field at every  point of $\partial \Omega$, then there exists a (global) outward-pointing vector field. Indeed, by compactness of $\partial \Omega$, we can find a finite family of points $\{x_1,\dots,x_N\}\subset \partial \Omega$ such that the open sets $U_{x_1},\dots,U_{x_N}$ cover $\partial \Omega$. Then,
\[F:=F_{x_1}+\cdots+F_{x_N}\]
is an outward-pointing vector field in the sense of Definition \ref{defOutward}.

To conclude the proof, we note that the existence of a local outward-pointing vector field at every point $x_0\in \partial \Omega$ follows easily from the definition of a Lipschitz domain. Up to a suitable choice of coordinates, we can assume that $x_0=0$ and that
\[\Omega\cap\left(B^{n-1}_r\times(-M,M)\right)=\left\{(x',y)\,:\, y<f(x')\right\},\]
where (with $r$ and $M$ some positive constants)
    $B^{n-1}_r$ is the ball centered at $0$ of radius $r$ in $\mathbb R^{n-1}$, and
  $f:B^{n-1}_r\to (-M,M)$ is a Lipschitz function such that $f(0)=0$.
We set 
\[F_0(x):=\chi(x)\,\left(\begin{array}{c} 0\\\vdots\\0\\1\end{array}\right),\]
with $\chi$ a smooth non-negative function such that
  $0\le \chi \le 1$ pointwise,
     the support of $\chi$ is contained in {$B^{n-1}_{r/2} \times (-M/2,M/2)$}, and
   $\chi=1$ pointwise in {$B^{n-1}_{r/4} \times (-M/4,M/4)$}.
The vector field $F_0$ is then locally outward-pointing at $0$.
\end{proof}

\begin{rem}
    By scaling, the outward-pointing vector field $F$ may be chosen to be unitary in a neighborhood of $\partial\Omega$, see e.g. \cite[Appendix A]{MT99}. 
\end{rem}

The existence of an outward-pointing vector field allows us to prove the fundamental inequalities used in this paper, which are a generalization of \cite[Proposition 2.2]{HS}.

\begin{prop}\label{prop:keyinequality}
    Suppose that $\Omega$ is a Lipschitz domain with outward-pointing vector field $F$. Suppose that $h\in L^{\infty}(\partial\Omega,\R)$, and that $\mu$ and $u$ are an eigenvalue and a corresponding eigenfunction respectively of the Robin Laplacian $\Delta_\Omega^{R,h}$. Finally, suppose that $D$ is a nodal domain of $u$. Then 
\begin{equation}\label{eqRayleigh}
\frac{\int_D|\nabla u|^2\, dx}{\int_D u^2\, dx} \le \big(\sqrt{\mu+\Gamma_1(\Omega,F)H} + \Gamma_2(\Omega,F)H\big)^2,
\end{equation}
where
\[H=\Vert\min\{0,h(x)\}\Vert_{L^\infty(\partial\Omega)}=\Vert \max\{-h(x),0\}\Vert_{L^{\infty}(\partial \Omega)}\]
and where
\begin{align}
    \Gamma_1(\Omega,F):=&\frac{1}{\gamma_F(\Omega)}\sup_{\Omega}\left|\nabla\cdot F\right|, \label{gam1}\\
    \Gamma_2(\Omega,F):=&\frac{2}{\gamma_F(\Omega)}\sup_{\Omega} \left|F\right|. \label{gam2}
\end{align}
\end{prop}

\begin{proof}
We assume for simplicity that $\partial D$ is Lipschitz and that $\partial D\cap\partial \Omega$ is an $(n-2)$-dimensional submanifold of $\partial D$. If not, one uses an approximation argument via Sard's theorem as in \cite{BM82,cL16,HS}.

Since $u$ vanishes on $\partial D\cap \Omega$, Green's identity implies \begin{align*}
    \int_{\partial D\cap\partial \Omega}u^2\,ds\le\ & \frac{1}{\gamma_F(\Omega)}\int_{\partial D\cap\partial \Omega} u^2(F\cdot \nu)\,ds =\frac{1}{\gamma_F(\Omega)}\int_{\partial D} \left(u^2\,F\right)\cdot \nu\,ds\\
       =\ &   \frac{1}{\gamma_F(\Omega)}\int_D \left(2\,u\nabla u\cdot F+\mbox{div}(F)\,u^2\right)\,dx\\
       \le\ &\frac{1}{\gamma_F(\Omega)}\int_D \left(2\,|u|\,\left|\nabla u\right|\,|F|+\left|\mbox{div}(F)\right|\,u^2\right)\,dx \\
       \le\ &\int_D \left(\Gamma_2(\Omega,F)\,|u|\,\left|\nabla u\right|+\Gamma_1(\Omega,F)\,u^2\right)\,dx\,,
 \end{align*}
 and from the Cauchy-Schwarz inequality, we obtain
    \begin{equation}\label{eqIneqGreen}
    \int_{\partial D\cap\partial \Omega} u^2\,ds\le \Gamma_1(\Omega,F)\int_D u^2\,dx+\Gamma_2(\Omega,F)\left(\int_Du^2\,dx\right)^{1/2}\left(\int_D \left|\nabla u\right|^2\,dx\right)^{1/2}.
\end{equation}

 The bound \eqref{eqRayleigh} follows from \eqref{eqIneqGreen} precisely as in \cite[Proposition 2.2]{HS}. Namely, applying Green's identity to $u$ in the nodal domain $D$,
 \begin{align*}\int_D|\nabla u|^2\, dx =\ & \int_D u\Delta u\, dx + \int_{\partial D}u(\nabla u\cdot \nu)\, ds\\
  =\ & \mu\int_D u^2\, dx - \int_{\partial D\cap\partial\Omega} hu^2\, ds\\
  \le\ & \mu\int_D u^2\, dx + H\int_{\partial D\cap\partial\Omega} u^2\, ds,
 \end{align*}
where we have used the facts that $u=0$ on $\partial D\cap\Omega$ and that $\nabla u\cdot \nu = -hu$ on $\partial\Omega$. Inserting the bound \eqref{eqIneqGreen} here and rearranging yields \eqref{eqRayleigh}.
\end{proof}

A particularly useful example of an outward-pointing vector field, for a domain $\Omega$ with \emph{smooth} boundary, is the gradient of the distance function to the boundary, see Example \ref{ex:out}. For more general $\Omega$, we may sometimes take the outward-pointing vector field to be the gradient of a replacement for the distance function, which we define here:

\begin{defi}
    Let $\Omega\subset\R^n$ be a bounded Lipschitz domain. We say that a $C^2$ function $g=g(x)$ defined on an open neighborhood of $\overline{\Omega}$ is an \emph{outward-pointing function} if there exists $\gamma_g>0$ for which $\nabla g\cdot \nu \geq 
    \gamma_g $ almost everywhere on $\partial\Omega$.
\end{defi}
\begin{rem}\label{rem:outfnct}
    If $g=g(x)$ is a smooth outward-pointing function, then its gradient $F=\nabla g$ is an outward-pointing vector field, with
    \[\Gamma_1(\Omega,F) = \frac{1}{\gamma_F(\Omega)}\sup_{\Omega}|\Delta g|,\]
    \[\Gamma_2(\Omega,F) = \frac{2}{\gamma_F(\Omega)}\sup_{\Omega}|\nabla g|.\]
\end{rem}
This can be useful in special cases to analyze the geometric dependence of the constants $\Gamma_1(\Omega,F)$ and $\Gamma_2(\Omega,F)$. 

It is an intriguing question whether every bounded Lipschitz domain admits an outward-pointing function.  We highlight a few important examples that admit an outward-pointing function, where the geometric dependency of these constants can be expressed more explicitly. 
\begin{ex}\label{ex:out}    
\begin{enumerate}[a)]
    \item Let $\Omega$ be an open, bounded, connected subset of $\R^n$, $n \geq 2$, with $C^2$ or $C^{1,1}$ boundary. Then we can define the outward-pointing function $g$ to be any $C^2$ function that coincides with the distance function $d(\cdot,\partial \Omega)$ on a neighborhood of the boundary staying within the positive reach of the boundary, that is a neighborhood of the boundary where there exists a unique nearest point in $\partial\Omega$.  When the boundary is $C^2$, the positive reach of $\partial \Omega$ is a neighborhood of the boundary defined by the cut-distance (the distance between $\partial \Omega$ and its cut locus). When the boundary is $C^{1,1}$, the fact that it has a positive reach is established in \cite[Theorem 4.12]{Fed59}. In both cases, the distance function from the boundary is $C^2$ or $C^{1,1}$, respectively. For the latter, see~\cite[Lemma 2.3]{HS}.   Then $\gamma_{\nabla g}(\Omega)=1$, and $\Gamma_1(\Omega,\nabla g)$ and $\Gamma_2(\Omega,\nabla g)$ can be expressed in terms of bounds on the Laplacian and the gradient of the distance function. See Lemma \ref{lem:bdsprop22} for an explicit bound for $C^2$ domains.
    
    In Section \ref{sec:uppernumber}, we use these types of bounds in order to obtain an upper bound for the largest Courant-sharp Robin eigenvalue for a $C^2$  convex domain, extending the result in \cite{GL18} to the case of the Robin problem where the parameter can be negative.
    \\

    \item Another important family of examples are curvilinear polygons. Let $\Omega$ be a curvilinear polygon (i.e., a planar domain with smooth boundary except for a finite number of vertices). In this case, we can also construct an outward-pointing function $g$ for which $\Gamma_1(\Omega,\nabla g)$ and $\Gamma_2(\Omega,\nabla g)$ depend only on a list of geometric quantities described in~\cite{BCM}. 
    
    In \cite{BCM},  Beck, Canzani, and Marzuola considered a family of \textit{chain domains}. They defined a chain domain as a collection of domains with smooth boundaries except for finitely many vertices that are joined by thin necks. They proved Pleijel's theorem for the Courant-sharp Neumann eigenvalues of chain domains as well as an upper bound for  Courant-sharp Neumann eigenvalues of such domains. Curvilinear polygons are a subfamily of chain domains that have no necks. 
    
    We outline the approach for constructing an outward-pointing function for curvilinear polygons.  For some $\delta>0$ which depends on some explicit geometric quantities, we take the (local) outward-pointing function to be $d(\cdot, \partial \Omega)$, in a $\delta$-neighborhood of each smooth arc of the boundary. In a $\delta$-neighborhood of each of the vertices, we view the angle as the graph $\Gamma f = \{(x,y): y = f(x)\}$ of a function $f$, where the vertex is located at $(0, f(0))$, and the $y$-axis is directed along the bisection of the angle. We define the (local) outward-pointing function to be $g(x,y) = y$. Then, in this neighborhood, $\gamma_{\nabla g}$ can be bounded from below by some explicit geometric quantities.

We can glue these functions together using a partition of unity. Roughly speaking, this gluing occurs in the region where the distance to each vertex is between $\delta/2$ and $\delta$. We arrange the partition of unity so that it depends only on the tangential coordinate along the boundary and not on the distance to the boundary. Consequently, the gradient of the partition of unity is parallel to the boundary and is bounded by $C\delta^{-1}$, while the Laplacian of the partition of unity is bounded by $C\delta^{-2}$, where $C$ is a universal constant. Moreover, since the normal derivatives of the partition functions vanish, the normal derivative of the outward-pointing function is bounded from below by $1$ away from the vertices. Near the vertices, the lower bound remains away from zero and can be explicitly given in terms of geometric quantities.

By using Proposition \ref{prop:keyinequality} it is possible to extend the results of \cite{BCM} for the Robin problem on curvilinear polygons. 
However, in the next section, we provide a different approach to prove Pleijel's theorem for the Robin problem in the more general setting of Lipschitz domains in $\R^n$, using Proposition \ref{prop:keyinequality} and adapting the techniques from~\cite{DFV}.
\end{enumerate}
\end{ex}

\section{Pleijel's Theorem for the Robin problem for Lipschitz domains in $\R^n$}\label{ss:pleij}

The proof of Pleijel's theorem for the Robin problem with an arbitrary Robin parameter $h\in L^\infty(\partial \Omega, \R)$ in \cite{HS} relies on the techniques from \cite{cL16}, which require the $C^{1,1}$ regularity of the domain, and an estimate relating the Robin eigenvalue on a nodal domain to the first mixed Dirichlet-Neumann eigenvalue of that domain. By using the techniques used in \cite{DFV}  together with Proposition \ref{prop:keyinequality}, we obtain Pleijel's theorem for Lipschitz domains. However, to obtain an improved version, the proof becomes more complex because we also need to establish a quantitative version of the Faber-Krahn inequality for the first mixed Dirichlet-Neumann eigenvalue on domains with small volume. Therefore, we include the proofs of both Pleijel's theorem and the improved version of Pleijel's theorem. 

We begin with Pleijel's theorem.

\begin{thm}\label{thm:robinplej}
  Let $\Omega\subset\R^n$ be a bounded Lipschitz domain. Then
\[\limsup_{k\to\infty}\frac{\nodal_\Omega^h(k)}{k}\le \gamma(n)<1,\]
where $\gamma(n):=\frac{(2\pi)^n}{\omega_n^2 j^n_{\frac{n-2}{2}}}.$
\end{thm}

We first state a {simplified} version of the results of  \cite[Theorem 5.1, and Theorem 5.3]{DFV} for open subsets of a Lipschitz domain in $\R^n$ in the following lemma. Note that Lipschitz domains are examples of uniform domains in $\R^n$ for which the results of  \cite{DFV} hold. These results 
give a version of the Faber-Krahn inequality for the mixed Dirichlet-Neumann problem on domains with small volume.  

Let $\Omega$ be an open Lipschitz domain in $\R^n$. For any open set $U\subset \Omega$, we consider the following mixed Dirichlet-Neumann problem.
\[\begin{cases}
    \Delta f=\lambda f& \text{in } U,\\
    \partial_\nu f=0 & \text{on } \partial U\cap \partial \Omega,\\
    f=0& \text{on } \partial U\cap \Omega.
\end{cases}\]
We denote its first eigenvalue by $\lambda_1(U)$.
When $U$ is compactly contained in $\Omega$, it is the Dirichlet eigenvalue problem on $U$. 

\begin{lemma}\label{prop:DFV}\begin{itemize}
\item[a)] There exist positive constants $c_1=c_1(\Omega)$ and $c_2=c_2(\Omega,n)$ such that for every open set $U\subset \Omega$ with $|U|\le c_1$, we have  
\[\lambda_1(U)|U|^{2/n}\ge c_2.\]
    \item[b)]
For every $\epsilon\in(0,1)$ and $\delta>0$, there {exists} a {neighborhood $\Omega_{\delta}$ of $\partial \Omega$} with $|\Omega_{\delta}|<\delta$, and constants $\vtheta_0=\vtheta_0(\Omega,n,\epsilon,\delta)$ and $\vtheta_1=\vtheta_1(\Omega,n,\epsilon)$ such that for any open set $U\subset \Omega$ with $$|U|\le \vtheta_0\quad \text{and}\quad \frac{|U\cap \Omega_{\delta}|}{|U|}\le \vtheta_1,$$
the following holds 
\begin{equation}\label{eq:FK}
\lambda_1(U)|U|^{2/n}\ge (1-{\epsilon})\lambda_1^D(\bb)|\bb|^{2/n},    
\end{equation}
where $\bb$ is the ball of radius $1$ in $\R^n$.

\end{itemize}
\end{lemma}
We assume without loss of generality that $\vtheta_0\le c_1$ throughout this section and in the Appendix.

\noindent
\begin{proof}[Proof of Theorem \ref{thm:robinplej}]
 Let $D$ be a nodal domain of $u_k$, an eigenfunction corresponding to the Robin eigenvalue $\mu_k$. Then combining with Inequality \eqref{eqRayleigh}, we have
    \begin{equation}\label{eq:lambda1}
        \lambda_1(D)\le \frac{\int_D |\nabla u_k|^2}{\int_D u_k^2}\le \big(\sqrt{\mu_k+\Gamma_1(\Omega,F)H} + \Gamma_2(\Omega,F)H\big)^2 = \mu_k + o(\mu_k). \end{equation}
        The last identity is for $k$ large enough so that $\mu_k>0$. Throughout the proof, we can assume this is the case.       
Let $\{D_j\}_{j=1}^{\nodal_\Omega^h(k)}$ be the nodal domains of $u_k$. For given $\epsilon\in(0,1), \delta\in(0,c_1)$, let $c_1$, $\Omega_\delta$, $\vtheta_0$ and $\vtheta_1$ be as in Lemma \ref{prop:DFV}. We now proceed as in \cite{DFV}, categorising the nodal domains into three disjoint classes as follows. For the reader's convenience, we include the details here.
\begin{itemize}
    \item[I.] $|D_j|> \vtheta_0$;
    \item[I\!I.]$|D_j|\le \vtheta_0$ and ${|D_j\cap \Omega_{\delta}|}> \vtheta_1 {|D_j|}$;
    \item[I\!I\!I.]$|D_j|\le \vtheta_0$ and ${|D_j\cap \Omega_{\delta}|}\le \vtheta_1 {|D_j|}$,
\end{itemize}
 Let $N_I$, $N_{\II}$ and $N_{\III}$ denote the number of nodal domains in each family respectively. Note that $\nodal_\Omega^h(k)= N_I+N_{I\!I}+N_{I\!I\!I}$.\\

 \noindent\textbf{Type I nodal domains.~} We clearly have $
  N_I\le |\Omega|/\vtheta_0.$ 
  Therefore, 
  \begin{equation}\label{eq:tI}
      \frac{N_I}{k}\le \frac{|\Omega|}{\vtheta_0 k}.
  \end{equation}

 \noindent\textbf{Type I\!I nodal domains.~}
By Lemma \ref{prop:DFV} and inequality \eqref{eq:lambda1}, we obtain the following inequality for $N_{I\!I}$.  
\begin{eqnarray*}
(\mu_k+o(\mu_k))^{n/2}\delta&\ge& (\mu_k+o(\mu_k))^{n/2}|\Omega_{\delta}|\\
  &\ge&\sum_{j\in I\!I}\lambda_1(D_j)^{n/2}|D_j\cap\Omega_{\delta}|
 \\
 &\ge& \vtheta_1\sum_{j\in \II}\lambda_1(D_j)^{n/2}|D_j|\\
 &\ge&\vtheta_1 c_2^{n/2} N_{\II};
\end{eqnarray*}
Therefore,
\begin{equation}\label{eq:tII}
    \frac{N_{I\!I}}{k}\le \frac{(\mu_k+o(\mu_k))^{n/2} \delta}{\vtheta_1 c_2^{n/2}k}.
\end{equation}
\noindent\textbf{Type I\!I\!I nodal domains.~} Again, by Lemma \ref{prop:DFV} and inequality \eqref{eq:lambda1}, we get 
\begin{equation*} 
  (\mu_k+o(\mu_k))^{n/2}|\Omega|
  \ge\sum_{j\in \III}\lambda_1(D_j)^{n/2}|D_j|
  \ge N_{\III} (1-\epsilon)^{n/2}\lambda_1^D(\bb)^{n/2}|\bb|.
\end{equation*}
Hence, 
\begin{equation}\label{eq:tIII}
    \frac{N_{\III}}{k}\le \frac{ (\mu_k+o(\mu_k))^{n/2}|\Omega|}{(1-\epsilon)^{n/2}\lambda_1^D(\bb)^{n/2}|\bb|k}.
\end{equation}
Taking the limit of \eqref{eq:tI}, \eqref{eq:tII}, and \eqref{eq:tIII} as $k\to\infty$ and using the  Weyl asymptotics for the Robin problem (see e.g. \cite{BS80,FG12}): 
\begin{equation}\label{weyl}
\lim_{k\to\infty}\frac{\mu_k^{\frac{n}{2}}|\Omega|}{k}=\frac{(2\pi)^n}{\omega_n},
\end{equation}
we get
\[
\limsup_{k\to\infty}\frac{\nodal_\Omega^h(k)}{k}=\frac{N_I+N_{\II}+N_{\III}}{k} \le
    \frac{\delta (2\pi)^n}{\vtheta_1 c_2^{n/2}\omega_n|\Omega|}+\frac{ \gamma(n)}{(1-\epsilon)^{n/2}}.\]
Note that $\omega_n j^n_{\frac{n-2}{2}}=\lambda_1^D(\bb)^{n/2}{|\bb|}$. 
We conclude by first sending $\delta\to0$ and then sending $\epsilon \to 0$. Note that $\vtheta_1$ is independent of~$\delta$.
\end{proof}

We now prove an improved version of Pleijel's theorem. See \cite{S14, B15} for improvements for the Dirichlet problem and \cite{HS} for an improvement for the Robin problem on $C^{1,1}$ domains.
\begin{thm}[Pleijel's theorem - Improved]\label{thm:pleijel improved} There exists a constant $\varepsilon=\varepsilon(n)>0$ depending only on the dimension,  such that for any bounded Lipschitz domain $\Omega\subset\R^n$ we have 
\[\limsup_{k\to\infty}\frac{\nodal_\Omega^h(k)}{k}\le \gamma(n)-\varepsilon.\]  
\end{thm}

We discuss the strategy of the proof as the details are rather intricate. To improve the upper bound $\gamma(n)$ in Theorem \ref{thm:robinplej}, we first need to prove a quantitative Faber-Krahn inequality for the mixed Dirichlet-Neumann eigenvalue for nodal domains with `small' volume. The results in \cite{DFV} do not yield a quantitative version. We prove this important component of the proof in the Appendix. Next, we categorize the nodal domains into disjoint families and use a sphere-packing argument to show that the family of nodal domains satisfying a quantitative version of the Faber-Krahn inequality constitutes a nontrivial proportion of the total volume. This allows the $\varepsilon$ improvement.

We first state a quantitative version of the Faber-Krahn inequality for which we introduce a slightly modified version of the Fraenkel asymmetry.

\begin{prop}\label{prop:fkq} For any $\epsilon\in (0,1)$, $\delta >0$ sufficiently small, and any $C_0<\infty$, there {exists} a neighbourhood $\Omega_{\delta}$ of $\partial\Omega$ with $|\Omega_{\delta}|<\delta$, a positive constant $\vtheta_0=\vtheta_0(\Omega, \epsilon,\delta)$, and another positive constant {$\vtheta_1=\vtheta_1(\Omega,n,\epsilon, C_0)$},
such that for any open set $D\subset \bOm$ satisfying 
\[|D|\le \vtheta_0,\qquad \frac{|D\cap\Omega_\delta|}{|D|}\le \vtheta_1,\quad \lambda_1(D)|D|^{2/n}\le C_0,\] 
we have 
\[\lambda_1(D)\ge (1-\epsilon+C\tilde A(D)^4)\lambda_1^D(D^*).\]
Here $C$ is a constant depending only on $n$, $D^*$ is a ball with the same volume as $D$ and $\tilde A(D)=\inf_U A(D\cap U)$, where $A$ is the Fraenkel asymmetry and the infimum is taken over all $U\subseteq\Omega$ containing $\Omega_\delta^c := \Omega\setminus\Omega_\delta$.
\end{prop}
We refer to the quantity $\tilde A(D)$ as a {\it modified Fraenkel asymmetry of $D$}. 
The proof of Proposition \ref{prop:fkq} is deferred to the Appendix. 
\begin{proof}[Proof of Theorem \ref{thm:pleijel improved}]
 Let $\{\mu_k\}$ be the Robin eigenvalues with a corresponding basis $\{u_k\}$ of Robin eigenfunctions. Abusing notation, we let $C$ refer to any positive constant, depending only on $n$ unless otherwise specified. Throughout, for $k$ large enough so that $\mu_k>0$, we let $B_k$ be a ball in $\mathbb R^n$, chosen so that its first Dirichlet eigenvalue is $\mu_k$:
\[\lambda_1^D(B_k)=\mu_k.\]

\begin{lemma}\label{prop:prep}
    There exist $k_0=k_0(\Omega)\in\mathbb N$ and $C>0$ for which, for all $k\ge k_0$,
\begin{equation}\label{eq:volbound}|B_k|\le C\frac{|\Omega|}k.\end{equation}
\end{lemma}

\begin{proof} By \eqref{weyl}, there exists $k_0$ depending on $\Omega$ such that, for $k \geq k_0$,
\begin{equation}\label{eq:aux1}
    \mu_k\ge C\left(\frac{k}{|\Omega|}\right)^{\frac{2}{n}},
\end{equation}
where $C=((2\pi)^n/(2\omega_n))^{n/2}$ and in particular depends only on $n$. However, by scaling,
\[\mu_k=\lambda^D_1(B_k)=\left(\frac{|\bb|}{|B_k|}\right)^{\frac 2n}\lambda^D_1(\bb) = C|B_k|^{-\frac 2n}.\]
Plugging this equation into \eqref{eq:aux1} gives the desired conclusion.
\end{proof}

Now we fix $\epsilon\in (0,1)$ and $\delta >0$ sufficiently small. (This ``sufficiently small" may depend only on the measure of $\Omega$ and on $n$. See the Appendix \pageref{page}.) We supplement our fixed $\epsilon$ and $\delta$ by fixing a set $\Omega_{\delta}$ and constants $\vtheta_0$ and $\vtheta_1$ which satisfy Proposition \ref{prop:fkq} {and Lemma~\ref{prop:DFV}}. We may also choose $\vtheta_1$ so that $\vtheta_1\le\frac 12$ and $\vtheta_1\le\epsilon_2$, where $\epsilon_2$ is defined below (note $\epsilon_2$ does not depend on $\vtheta_1$, so that the argument is not circular).

Finally, we fix $\epsilon_1,\epsilon_2\in(0,1)$ sufficiently small and independent of both $\epsilon$ and $\delta$ -- see below. For each $k$ satisfying
\begin{equation}\label{eq:kcdn}
k\ge\max\left\{k_0,\frac{2}{\vtheta_0}C|\Omega|\right\},
\end{equation}
with $C$ as in Lemma \ref{prop:prep},
we categorise the nodal domains $D_j$ of the eigenfunction $u_k$ into four disjoint sets (type-I, type-\II, type-\III, and type-\IV) as follows.
\begin{itemize}
    \item[\I.] $|D_j|>(1+\epsilon_1)|B_k| $;
    \item[\II.]$|D_j|\le (1+\epsilon_1)|B_k| $ and ${|D_j\cap \Omega_{\delta}|}> \vtheta_1 {|D_j|}$;
    \item[\III.]$|D_j|\le(1+\epsilon_1)|B_k| $, ${|D_j\cap \Omega_{\delta}|}\le \vtheta_1 {|D_j|}$, and  $\tilde A(D_j)\le \epsilon_2$;
    \item[\IV.] $|D_j|\le (1+\epsilon_1)|B_k| $, ${|D_j\cap \Omega_{\delta}|}\le \vtheta_1 {|D_j|}$, and  $\tilde A(D_j)> \epsilon_2$.
\end{itemize}
 Let $N_\I$, $N_{\II}$, $N_{\III}$, and $N_{\IV}$ denote the number of nodal domains in each family. Note that $\nodal_\Omega^h(k)= N_\I + N_{\II} + N_{\III} +N_{\IV}$. We denote by $\Omega_{\sharp}$ the union of all nodal domains of type $\sharp$. Note also that our condition \eqref{eq:kcdn} on $k$ guarantees that 
 \[|B_k|\le\min\left\{C\frac{|\Omega|}{k},\frac{\vtheta_0}{2}\right\}.\]

 \begin{rem}
 We can now be specific about our choice of $\epsilon_1$ and $\epsilon_2$. They are chosen independent of (sufficiently small) $\epsilon$, and in a way such that type-{\III} domains can only pack a $\tilde\rho(n)<1$ fraction of $\Omega_{\delta}^c$; we show in our proof of Lemma \ref{lem:spherepacking} below that this is possible.
 \end{rem}

 Our first step is to show that type-{\III} domains do not have density 1 -- that is, that the fraction of nodal domains which are \emph{not} type-{\III} is bounded away from zero. This is a sphere-packing argument.

\begin{lemma}\label{lem:spherepacking}
    There exist $k_1=k_1(\delta,\epsilon,\epsilon_1,\epsilon_2,\bOm)\in\mathbb N$ and $\tilde\rho=\tilde\rho(n) < 1$ such that for all $k\ge k_1$,
    \[\frac{|\Omega_{\III}|}{|\Omega|} \le\tilde\rho.\]
\end{lemma}
To prove this lemma, we begin with an auxiliary result.
\begin{lemma}\label{lem:modFrankel}
    Suppose that $D$ is a type-{\III} nodal domain. Then
    \begin{equation}\label{eq:set}
        A(D\cap\Omega_{\delta}^c) \le \frac{\epsilon_2}{1-\vtheta_1} + \frac{2\vtheta_1}{1-\vtheta_1}\le 6\epsilon_2.
    \end{equation}
\end{lemma}
\begin{proof}
We claim a slightly more general result: if $S$ and $T$ are domains with $A(S)<\epsilon_2$, $T\subseteq S$, and $|T|/|S|\ge 1-\vtheta_1$,
then
\begin{equation}\label{eq:genversion}
    A(T)\le\frac{\epsilon_2}{1-\vtheta_1}+\frac{2\vtheta_1}{1-\vtheta_1}.
\end{equation}
Assuming this claim for the moment, we know because $D$ is a type-{\III} nodal domain that there exists a $U$ with $\Omega_{\delta}^c\subseteq U\subseteq \Omega$ for which $A(D\cap U)<\epsilon_2$. We then apply \eqref{eq:genversion} with $S=D\cap U$ and $T=D\cap\Omega_{\delta}^c$, immediately proving the first inequality in \eqref{eq:set}. The last inequality in \eqref{eq:set} follows immediately from the inequalities $\vtheta_1\le\frac 12$ and $\vtheta_1\le\epsilon_2$.

It remains to prove \eqref{eq:genversion}. By definition of the Fraenkel asymmetry, there is a ball $B$ such that $|B|=|S|$ and $|B\Delta S|<\epsilon_2|B|$. Let $\tilde B$ be a ball of volume $|T|$ with $\tilde B\subseteq B$. Then
\[|\tilde B\Delta T| = |\tilde B\setminus T| + |T\setminus\tilde B|\le |B\setminus T|+|S\setminus\tilde B|\le (|B\setminus S|+|S\setminus T|) + (|S\setminus B|+|B\setminus \tilde B|).\]
Now $|S\setminus T|\le \vtheta_1|S|$ and the same is true for $|B\setminus\tilde B|$. So
\[|\tilde B\Delta T|\le |B\Delta S| + 2\vtheta_1|S|.\]
Dividing by $|\tilde B|=|T|$ gives
\[A(T)\le\frac{|\tilde B\Delta T|}{|\tilde B|} \leq \frac{1}{|S|}\frac{|S|}{|T|}(|B\Delta S|+2\vtheta_1|S|),\]
and now \eqref{eq:genversion} follows from the assumptions on $S$ and $T$.
\end{proof}

\begin{proof}[Proof of Lemma \ref{lem:spherepacking}] Let $D$ be a type-{\III} nodal domain associated with $\mu_k$. By Lemma \ref{prop:DFV} part b),  we have
\[|D|^{2/n}\lambda_1(D)\ge(1-\epsilon)\lambda_1^D(\bb)|\bb|^{2/n} = (1-\epsilon)\lambda_1^D(B_k)|B_k|^{2/n}=(1-\epsilon)\mu_k|B_k|^{2/n}.\]
By inequality \eqref{eq:lambda1}, there exists $\tilde k_1$ depending only on $\Omega$ and $\epsilon$ for which $k\ge \tilde k_1$ implies
\[(1-\epsilon)\mu_k|B_k|^{2/n}\le \lambda_1(D)|D|^{2/n}\le(1+\epsilon)\mu_k|D|^{2/n}.\]
Therefore, also using the fact that $D$ is a type-{\III} nodal domain,
\begin{equation}\label{eq:uplowD}\left(\frac{1-\epsilon}{1+\epsilon}\right)^{n/2}|B_k|\le |D|\le (1+\epsilon_1)|B_k|.\end{equation}

Now, suppose that $D$ and $\tilde D$ are any two nodal domains associated to $\mu_k$. Let $B$ and $\tilde B$ be balls such that
\[|B|=|D\cap\Omega_{\delta}^c|,\quad |\tilde B|=|\tilde D\cap\Omega_{\delta}^c|.\]
By Lemma \ref{lem:modFrankel}, the centers of $B$ and $\tilde B$ can be chosen so that
\[|B\Delta(D\cap\Omega_{\delta}^c)|\le 6\epsilon_2|B|,\quad |\tilde B\Delta(\tilde D\cap\Omega_{\delta}^c)|\le 6\epsilon_2|\tilde B|.\]
We claim that $B$ and $\tilde B$ have comparable volume and small overlap.

To show that $B$ and $\tilde B$ have comparable volume, first observe that by the definition of a type-{\III} nodal domain, then using \eqref{eq:uplowD} for both $D$ and $\tilde D$, as well as the fact that $\vtheta_1\le\epsilon_2$,
\begin{equation}
\frac{|B|}{|\tilde B|}=\frac{|D\cap\Omega_\delta^c|}{|\tilde D\cap\Omega_{\delta}^c|}\le\left(\frac{1+\vtheta_1}{1-\vtheta_1}\right)\frac{|D|}{|\tilde D|}\le\left(\frac{1+\epsilon_2}{1-\epsilon_2}\right)(1+\epsilon_1)\left(\frac{1+\epsilon}{1-\epsilon}\right)^{n/2}.\end{equation}
We can assume $\epsilon\le\epsilon_2\le \epsilon_1\le \frac{1}{2}$. Thus,
\begin{equation}\label{eq:volumecomp}
\frac{|B|}{|\tilde B|}\le{(1+\epsilon_1)\left(1+\frac{2\epsilon_1}{1-\epsilon_1}\right)^{(n+2)/2}}\le 1+C\epsilon_1.
\end{equation} The same upper bound holds for $|\tilde B|/|B|$. 

To show that $B$ and $\tilde B$ have small overlap, observe that since $D$ and $\tilde D$ are disjoint,
\[|B\cap\tilde B|\le |B\Delta(D\cap\Omega_{\delta}^c)| + |\tilde B\Delta(\tilde D\cap\Omega_{\delta}^c)|\le 6\epsilon_2(|B|+|\tilde B|),\]
and so using \eqref{eq:volumecomp},
\begin{equation}\label{eq:smalloverlap}\frac{|B\cap\tilde B|}{|B|}\le 6\epsilon_2\left(1+\frac{|\tilde B|}{| B|}\right)\le 6\epsilon_2\left(2+C\epsilon_1\right)\le C\epsilon_2,\end{equation}
where again by abuse of notation, $C$ in the left-hand side of the last inequality is equal to $6(2+{C})$.

Now consider such a ball for \emph{each} type-{\III} nodal domain of $\mu_k$. Note that the volume of each of these balls is less than $\frac{C|\Omega|}{k}$. Let $\rho_\circ = \rho_\circ(n) < 1$ represent the sphere packing density of $\mathbb{R}^n$. The sphere packing density of $\Omega$ is bounded above by $\rho_\circ$. Therefore, there exists a threshold $\zeta>0$ for which the packing density of $\Omega$ with balls of any fixed size less than $\zeta$ is bounded above by $(\rho_\circ+1)/2$. We choose $k$ sufficiently large so that $\frac{C|\Omega|}{k}<\zeta$. By choosing $\epsilon_1$ to be sufficiently small, the balls will have almost the same radius. We then shrink the balls by a controlled factor, which depends only on $\epsilon_1$ and $n$, so that they all have the same radius.
Next, we select $\epsilon_2$ to be small enough so that their overlap is very small. By shrinking the balls again by a controlled factor, depending only on $\epsilon_2$, we can make them disjoint, thereby obtaining a sphere packing of $\Omega_{\delta}^c$ -- indeed,  a sphere packing of $\Omega$ itself.  Therefore,
\[(1-c(n,\epsilon_1,\epsilon_2))\sum_{j\in \III}|D_j\cap\Omega_{\delta}^c| <\frac{1+\rho_\circ}{2}|\Omega|,\]
where $c(n,\epsilon_1,\epsilon_2)$ is a positive function and goes to zero when $\epsilon_1$ and $\epsilon_2$ tend to zero. Hence, we can choose $\epsilon_1$ and $\epsilon_2$ small enough such that $\frac{(1+\rho_\circ)/2}{1-c(n,\epsilon_1,\epsilon_2)}<(1-\epsilon_2)\tilde \rho$, where $\tilde\rho$ is a fixed constant in the interval $(\rho,1)$.
Therefore since we have type-{\III} nodal domains,
\[|\Omega_{\III}|=\sum_{j\in \III}|D_j|\le \sum_{j\in \III}\frac{1}{1-\vtheta_1}|D_j\cap\Omega_{\delta}^c|\le\frac{1}{1-\epsilon_2}\sum_{j\in \III}|D_j\cap\Omega_{\delta}^c|\le\tilde\rho|\Omega|.\]
This completes the proof of Lemma \ref{lem:spherepacking}.\end{proof}

We now proceed to bound, for each $\sharp\in\{$\I, \II, \III, \IV$\}$, the quantity $\frac{N_{\sharp}}{k}$.
From adding the bounds we obtain and taking the limsup as $k\to\infty$, we will be able to read off our sharpened Pleijel theorem.\\

\textbf{Type-{\I} nodal domains.}
Suppose that $D_j$ is a nodal domain of type I. Then it is immediate from the definition of type {\I} that
\[\lambda_1^D(B_k)^{n/2}|B_k|(1+\epsilon_1)=\mu_k^{n/2}|B_k|(1+\epsilon_1)\le\mu_k^{n/2}|D_j|.\]
Summing this equation over all nodal domains of type I yields
\[N_{I}\lambda_1^D(B_k)^{n/2}|B_k|(1+\epsilon_1)\le\mu_{k}^{n/2}|\Omega_I|,\]
and thus
\begin{equation}\label{eq:countI}\frac{N_\I}{k}\le\frac{\mu_k^{n/2}|\Omega|}{\lambda_1^D(\bb)^{n/2}|\bb|k}\cdot\frac{|\Omega_I|}{|\Omega|(1+\epsilon_1)}.\end{equation}
{\ }\\

\textbf{Type-{\II} nodal domains.}We do the same treatment as for type-{\II} nodal domains in the proof of Theorem \ref{thm:robinplej} to obtain
\begin{equation}\label{eq:countII}
    \frac{N_{I\!I}}{k}\le \frac{(\mu_k+o(\mu_k))^{n/2} \delta}{\vtheta_1 c_2(n,\Omega)^{n/2}k},
\end{equation}
where $c_2(n,\Omega)$ is the constant in Lemma \eqref{prop:DFV} (a).
{\ }\\

\textbf{Type-{\III} nodal domains.} Type-{\III} nodal domains satisfy the assumption of Lemma \ref{prop:DFV} (b). Thus, we have 
\[\lambda_1(D_j)|D_j|^{2/n}\ge (1-\epsilon)\lambda_1^D(\bb) |\bb|^{2/n}.\]
By inequality  \eqref{eq:lambda1},
\begin{equation*}   
(\mu_k+o(\mu_k))^{n/2}|\Omega_{\III}| \ge \sum_{j\in \III}\lambda_1(D_j)^{n/2}|D_j|
\ge (1-\epsilon)^{n/2}\lambda_1^D(\bb)^{n/2}|\bb|N_{\III}.\end{equation*}
Solving for $N_{\III}$ and dividing by $k$, we see as in the proof for Type I nodal domains that
\begin{equation}\label{eq:countIII}
    \frac{N_{\III}}{k}\le\frac{(\mu_k+o(\mu_k))^{n/2}|\Omega|}{\lambda_1^D(\bb)^{n/2}|\bb|k}\cdot\frac{|\Omega_{\III}|}{(1-\epsilon)^{n/2}|\Omega|}.
\end{equation}
{\ }\\

\textbf{Type-{\IV} nodal domains.}
The analysis of type-{\IV} nodal domains is very similar to the analysis of type-{\III} nodal domains, but now we have a lower bound on $\tilde A(D_j)$. Moreover, by Lemma \ref{prop:prep} and inequality \eqref{eq:lambda1}, for $k\ge k_0$, using also the trivial bound $\epsilon_1\le 1$,
\[\lambda_1(D_j)|D_j|^{2/n}\le C\left(\frac{|\Omega|}{k}\right)^{2/n}(\mu_k + o(\mu_k)),\]
which by Weyl's law is bounded by $C$. Hence Proposition \ref{prop:fkq} applies taking $C_0=C$, and therefore 
\[\lambda_1(D_j)\ge (1-\epsilon+C\epsilon_2^4)\lambda_1^D(D_j^*).\]
Taking $\epsilon$ small enough such that $2\epsilon\le C\epsilon^4_2$ and applying the same logic as for type-{\III}, we obtain
\begin{equation}\label{eq:countIV}
    \frac{N_{\IV}}{k}\le\frac{(\mu_k+o(\mu_k))^{n/2}|\Omega|}{\lambda_1^D(\bb)^{n/2}|\bb|k}\cdot\frac{|\Omega_{IV}|}{(1+C\epsilon_2^4)^{n/2}|\Omega|}.
\end{equation}
Note that in \eqref{eq:countIV}, by abuse of notation $C=\frac{C}{2}$.  
{\ }\\

\textbf{Completing the proof.} Combining \eqref{eq:countI}, \eqref{eq:countII}, \eqref{eq:countIII}, and \eqref{eq:countIV} yields
\begin{multline}\label{eq:count}
\frac{{\nodal_\Omega^h(k)}}{k}\le\frac{\delta}{|\Omega|\vtheta_1 C}\frac{(1+o(1))\mu_k^{n/2}|\Omega|}{k} + \frac{(1+o(1))\mu_k^{n/2}|\Omega|}{k|\bb|\lambda_1^D(\bb)^{n/2}}\Big(\frac{1}{1+\epsilon_1}\frac{|\Omega_{\I}|}{|\Omega|} + \\ \frac{1}{(1-\epsilon)^{n/2}}\frac{|\Omega_{\III}|}{|\Omega|} + \frac{1}{(1+C\epsilon_2^4)^{n/2}}\frac{|\Omega_{\IV}|}{|\Omega|}\Big).
\end{multline}
Now take the limsup. Using Weyl's law yields
\begin{equation}\label{eq:limsup}
\limsup_{k\to\infty}\frac{{\nodal_\Omega^h(k)}}{k}\le\frac{C\delta}{\vtheta_1|\Omega|} + \gamma(n)\limsup_{k\to\infty}\Big(\frac{1}{1+\epsilon_1}\frac{|\Omega_{\I}|}{|\Omega|} + \frac{1}{(1-\epsilon)^{n/2}}\frac{|\Omega_{\III}|}{|\Omega|} + \frac{1}{(1+C\epsilon_2^4)^{n/2}}\frac{|\Omega_{\IV}|}{|\Omega|}\Big).
\end{equation}
 Pulling out the $(1-\epsilon)^{-n/2}$ and choosing $\epsilon_3>0$, again independent of $\epsilon$ and $\delta$, so that 
\[(1-\epsilon_3)\ge \max\left\{\frac{1}{1+\epsilon_1},\frac{1}{(1+C\epsilon_2^4)^{n/2}}\right\}\]
gives
\begin{equation}\label{eq:limsup2}
\limsup_{k\to\infty}\frac{{\nodal_\Omega^h(k)}}{k}\le\frac{C\delta}{\vtheta_1|\Omega|} + \frac{\gamma(n)}{(1-\epsilon)^{n/2}}\limsup_{k\to\infty}\Big((1-\epsilon_3)\frac{|\Omega_I|+|\Omega_{\IV}|}{|\Omega|} + \frac{|\Omega_{\III}|}{|\Omega|}\Big),
\end{equation}
which in turn means
\begin{equation}\label{eq:limsup3}
\limsup_{k\to\infty}\frac{{\nodal_\Omega^h(k)}}{k}\le\frac{C\delta}{\vtheta_1|\Omega|} + \frac{\gamma(n)}{(1-\epsilon)^{n/2}}\limsup_{k\to\infty}\left((1-\epsilon_3)\left(1-\frac{|\Omega_{\III}|}{|\Omega|}\right) + \frac{|\Omega_{\III}|}{|\Omega|}\right).
\end{equation}
By Lemma \ref{lem:spherepacking}, for $k\ge k_1$, $|\Omega_{\III}|/|\Omega|\le\tilde\rho$. Therefore
\begin{equation}\label{eq:limsup4}
\limsup_{k\to\infty}\frac{{\nodal_\Omega^h(k)}}{k}\le\frac{C\delta}{\vtheta_1|\Omega|} + \frac{\gamma(n)}{(1-\epsilon)^{n/2}}\Big((1-\epsilon_3)(1-\tilde\rho)+\tilde\rho\Big).
\end{equation}
Now let $\delta\to 0$, using the fact that $\vtheta_1$, $\epsilon_3$, and $\tilde\rho$ are independent of $\delta$, we see
\begin{equation*}\label{eq:limsup5}
\limsup_{k\to\infty}\frac{{\nodal_\Omega^h(k)}}{k}\le \frac{\gamma(n)}{(1-\epsilon)^{n/2}}\Big((1-\epsilon_3)(1-\tilde\rho)+\tilde\rho\Big).
\end{equation*}
Finally, we let $\epsilon\to 0$, using that $\epsilon_3$ and $\tilde\rho$ are independent of $\epsilon$, and we observe that $(1-\epsilon_3)(1-\tilde\rho)+\tilde\rho$ is strictly less than 1 {and depends only on $n$}. This completes the proof.
\end{proof}

\section{Geometric upper bounds for the number of Courant-sharp Robin eigenvalues}\label{sec:uppernumber}

The goal of this section is to obtain an upper bound for the number of Courant-sharp Robin eigenvalues of an open, bounded, connected, convex set $\Omega \subset \R^n$, $ n \geq 2$, with $C^2$ boundary, that is explicit in terms of the geometric quantities of $\Omega$ and the Robin parameter.

To do this, we first derive a comparison between the Robin and Neumann counting functions for open, bounded, connected, Lipschitz sets in $\R^n$. We then derive an upper bound for the largest positive Courant-sharp Robin eigenvalue of an open, bounded, connected set $\Omega \subset \R^n$ with $C^2$ boundary.
By combining these two results with the result from \cite[Appendix A]{GL18} which holds for convex sets, we obtain an upper bound for the number of Courant-sharp Robin eigenvalues of $\Omega$. We consider the case where $\Omega$ has $C^2$ boundary as this setting allows us to glean explicit geometric control in the geometric bounds that follow.

\subsection{Comparison between Neumann and Robin counting functions} 

In order to use previous work estimating the Neumann eigenvalues and the Neumann counting function, we prove, for an arbitrary Lipschitz domain, a comparison result between the Neumann and the Robin spectra. This can be stated either as a lower bound for the Robin eigenvalues or an upper bound for the Robin counting function.
For $\mu > 0$, we define the Robin counting function as
$$ N_\Omega^h(\mu) :=\sharp\{k\in\N\,:\,\mu_k(\Omega,h)<\mu\}.$$
We denote the Neumann counting function by $N_\Omega^N(\mu) = N_\Omega^0(\mu)$.

We recall that for any Lipschitz domain $\Omega$ with outward-pointing vector field $F$ and for all $u\in H^1(\Omega)$, we have
\begin{equation}\label{eqTrace}
    \int_{\partial \Omega}u^2\, ds \le \Gamma_1(\Omega,F)\,\int_{\Omega}u^2\,dx+\Gamma_2(\Omega,F)\left(\int_{\Omega}u^2\,dx\right)^{\frac12}\left(\int_{\Omega}\left|\nabla u\right|^2\,dx\right)^{\frac12},
\end{equation}
with $\Gamma_1(\Omega,F)$ and $\Gamma_2(\Omega,F)$ the constants defined in Proposition \ref{prop:keyinequality}.

This can be interpreted as a kind of trace inequality, and it is the starting point of our analysis.

\begin{prop}\label{propComparisonNR}
Let $\Omega$ be a Lipschitz domain and let $h \in L^{\infty}(\partial \Omega)$ be  a real-valued function defined on its boundary $\partial \Omega$. Let $F$ be an outward-pointing vector field relative to $\Omega$, and let $H:=\|h^-\|_{L^{\infty}(\partial \Omega)}$ where 
$h^-=\max\{0,-h\}$. Then, setting
\begin{align*}
    K_1(\Omega,F)&:=\Gamma_1(\Omega,F);\\
    K_2(\Omega,F)&:=\frac{\Gamma_2(\Omega,F)^2}{4};
\end{align*}
(with $\Gamma_1(\Omega,F)$, $\Gamma_2(\Omega,F)$ defined in \eqref{gam1}, \eqref{gam2} respectively), 
for any $\eta\in(0,1)$ and any $k\ge1$, we have
\begin{equation}\label{eqEVRN}
    \mu_k(\Omega,h)\ge (1-\eta)\mu_k^N(\Omega)-\left(K_1(\Omega,F)H+K_2(\Omega,F)\frac{H^2}{\eta}\right).
\end{equation}
Equivalently, we have
\begin{equation}\label{eqCountRN}
    N_{\Omega}^h(\mu)\le N_{\Omega}^N\left(\frac{1}{1-\eta}\left(\mu+K_1(\Omega,F)H+K_2(\Omega,F)\frac{H^2}{\eta}\right)\right)
\end{equation}
for all $\mu\in\mathbb R$.
\end{prop}

\begin{proof} Without loss of generality we can assume that $H>0$. It follows immediately from Inequality \eqref{eqTrace} that, for all $u\in H^1(\Omega)$ and for any parameter $A>0$,
{
\begin{equation*} 
\int_{\partial \Omega}u^2\, ds \le\Gamma_1(\Omega,F)\,\int_{\Omega}u^2\,dx+\frac{\Gamma_2(\Omega,F)\,A}{2}\int_{\Omega}u^2\,dx+\frac{\Gamma_2(\Omega,F)}{2\,A}\int_{\Omega}\left|\nabla u\right|^2\,dx.
\end{equation*}}
Choosing $A=(2\eta)^{-1}{\Gamma_2(\Omega,F)} H$, we obtain
\begin{equation}\label{eql2normbdry}
    \int_{\partial \Omega}u^2\, ds \le \frac{\eta}{H}\int_{\Omega}\left|\nabla u\right|^2\,dx+\left(K_1(\Omega,F)+K_2(\Omega,F)\frac{H}{\eta}\right)\int_{\Omega}u^2\,dx.
\end{equation}
Inequality \eqref{eql2normbdry} implies
\begin{equation}\label{eqHTrace}
\begin{split}
    \int_{\partial \Omega} h\,u^2\,ds&\ge -H\int_{\partial \Omega} u^2\,ds\\&\ge  -\eta\int_{\Omega}\left|\nabla u\right|^2\,dx-\left(K_1(\Omega,F)H+K_2(\Omega,F)\frac{H^2}{\eta}\right)\int_{\Omega}u^2\,dx.
    \end{split}
\end{equation}

Now let $q_h$ denote the quadratic form associated to the Robin problem with boundary function $h$:
\begin{equation}\label{eqqh}
    q_h[u]:=\int_{\Omega}\left|\nabla u\right|^2\,dx+\int_{\partial \Omega}h\,u^2\,ds.
\end{equation}
It is closed with domain $H^1(\Omega)$. From the min-max principle, we have
\begin{equation}\label{eqminmaxR}
    \mu_k(\Omega,h)=\min_{V_k}\max_{u\in V_k\setminus\{0\}}\frac{q_h(u)}{\|u\|_{L^2(\Omega)}^2},
\end{equation}
where $V_k$ runs over all $k$-dimensional subspaces of $H^1(\Omega)$. In particular, Equation \eqref{eqminmaxR} also holds for $h=0$, that is, for the Neumann eigenvalue problem.

It follows immediately from \eqref{eqHTrace} that, for all $u\in H^1(\Omega)$,
\begin{equation}\label{eq:ineqqh}
    q_h[u]\ge(1-\eta)q_0(u)-\left(K_1(\Omega,F)H+K_2(\Omega,F)\frac{H^2}{\eta}\right)\|u\|_{L^2(\Omega)}^2.
\end{equation}
From the min-max principle, we get, {for all $k\ge1$,}
\begin{equation*}
    \mu_k(\Omega,h)\ge (1-\eta)\mu_k^N(\Omega)-\left(K_1(\Omega,F)H+K_2(\Omega,F)\frac{H^2}{\eta}\right).
\end{equation*}
This is Inequality \eqref{eqEVRN}. Inequality \eqref{eqCountRN} is straightforwardly equivalent.
\end{proof}

We see from the previous result that to obtain an upper bound for the Robin counting function, it is sufficient to {find} an upper bound for a corresponding Neumann counting function. {This has been done in \cite{GL18}, under the additional assumption that $\Omega$ is convex, to derive explicit upper bounds for the number of Courant-sharp Neumann eigenvalues. The proof of the following result is given in \cite[Appendix A]{GL18}.}

\begin{prop}\label{propCountVol}
For any convex domain $\Omega$ and any $\mu>0$,
\begin{equation}\label{eqCountVol}
    N^N_{\Omega}(\mu)\le \frac{n^{\frac{n}{2}}}{\pi^n}\,\mu^{\frac{n}{2}}\left|\Omega+\frac{\pi}{\sqrt{\mu}}\mathbb B\right|.
\end{equation}
where the last factor in the right-hand side is the volume (i.e. the Lebesgue measure) of the Minkowski sum of the two convex sets.
\end{prop}

From Inequality \eqref{eqCountVol} we see that an upper bound for the volume $\left|\Omega+\delta\,\mathbb B\right|$ translates into an upper bound for the Neumann counting function. 

{
When $\Omega$ is of class $C^2$ in addition to being convex, this volume can be estimated from above using the maximal scalar curvature of $\partial \Omega$. We get the following result.

\begin{cor}\label{corCountUpper} 
For any convex domain $\Omega$  of class $C^2$, for any $\mu>0$,
\begin{equation}
    \label{eq:CountUpper2}
    N^N_{\Omega}(\mu)\le  \frac{n^{\frac{n}{2}}}{\pi^n}|\Omega|\,\mu^{\frac{n}{2}}
    +n^{\frac{n}{2}}|\partial\Omega|\sum_{j=0}^{n-1}
    \frac{1}{j+1}
    \left(\begin{array}{c} n-1 \\ j\end{array}\right)\kappa_{\max}^j\left(\frac{\mu}{\pi^2}\right)^{\frac{n-j-1}{2}},
\end{equation}
where $\kappa_{\max}$ is the maximum over $\partial \Omega$ of the largest principal curvature. 
\end{cor}

\begin{proof} It was proved in  \cite[Appendix A]{GL18} that, for any $\delta>0$,
\begin{equation*}
    \left|\Omega+\delta \mathbb B\right|\le |\Omega|
    +|\partial \Omega|\sum_{j=0}^{n-1}
    \frac{1}{j+1}
    \left(\begin{array}{c} n-1 \\ j\end{array}\right)\left(\kappa_{\max}\right)^j\delta^{j+1}.
\end{equation*}
Taking $\delta=\frac{\pi}{\sqrt{\mu}}$ and using Proposition \ref{propCountVol}, we obtain \eqref{eq:CountUpper2}.
\end{proof}
}

Thus, in order to obtain an upper bound for the number of Courant-sharp Robin eigenvalues of $\Omega$, it is sufficient to obtain an upper bound for the largest Courant-sharp Robin eigenvalue and substitute it into the bounds for the counting functions using Proposition \ref{propComparisonNR} and Corollary \ref{corCountUpper}.

\subsection{Geometric upper bound for Courant-sharp Robin eigenvalues}\label{ss:geombdeval}
In this section we take $\Omega \subset \R^n$, $n \geq 2$, to be an open, bounded, connected set with $C^2$ boundary. We will see that the assumption that the boundary is $C^2$ allows us to obtain explicit geometric control in the desired bounds.

Throughout we use the following notation 
\begin{itemize}
    \item $V:=|\Omega|$ is the Lebesgue measure,
    \item $S:=|\partial \Omega|$ is the surface measure,
    \item $\rho:=S/V^{1-\frac1n}$ is the isoperimetric ratio,
    \item $t_+$ is the minimal radius of curvature (i.e., $t_{+}^{-1}$ is the supremum of the maximum modulus of the principal curvatures $\{\kappa_1, \dots, \kappa_{n-1}\}$ of $\partial \Omega$),
    \item $\delta_0$, is the minimum between $t_+$ and the cut-distance with respect to the interior of $\Omega$ (see, e.g., \cite[Section 3]{GL18} for a precise definition of the cut-distance),
    \item $\delta_1$ is the minimum between $t_+$ and the cut-distance with respect to the entire complement of $\partial \Omega$ (i.e., the interior and exterior). 
    \item $H:=\max\{-h(x)\,:\,x\in \partial \Omega\}$, where $h$ is the real-valued function appearing in the Robin boundary condition.
\end{itemize}

Our main result is the following.
\begin{thm} \label{thm:eig} 
Let $\Omega \subset \R^n$, $n \geq 2$, be an open, bounded, connected set with $C^2$ boundary. There exists a constant $C$, depending only on $n$, such that any Courant-sharp Robin eigenvalue $\mu$ satisfies
\[\mu\le C\left(\frac{V^{\frac2n}}{\delta_1^4}+\frac{\rho^4}{V^{\frac{2}{n}}}+V^{\frac2n}H^4\right).\]
\end{thm}
From the proof of Theorem \ref{thm:eig}, we will see that the constant $C$ may be very large. 
{We note that the result of Theorem \ref{thm:eig} is an extension of the results of \cite[Propositions 7.1 and 9.2]{GL18} from the case where $h \geq 0$ to the case where $h$ can be negative. It would be possible to apply the same arguments that are used in the proof of Theorem \ref{thm:eig} below to \cite[Propositions 7.1 and 9.2]{GL18} to obtain an analogous result to that of Theorem \ref{thm:eig} where the price for having a more explicit dependence on the geometric quantities would be paid via a large constant $C$. Thus, the key novelty in Theorem \ref{thm:eig} is that it holds in the case where $h$ can be negative.}

We first need to estimate the constants $\Gamma_1(\Omega,F)$, $\Gamma_2(\Omega,F)$ from Proposition \ref{prop:keyinequality} in terms of $n$ and of some of the geometric quantities of $\Omega$. To do this we will construct an outward-pointing function and use it to obtain an outward-pointing vector field. We have the following.

\begin{lemma}\label{lem:bdsprop22}
    {Let $n\geq 2$. There exists a constant $C>0$, depending only on $n$, such that, for any set $\Omega \subset \R^n$, open, bounded, connected with $C^2$ boundary,
     we can find an outward-pointing vector field $F$ for which }
    $$ \Gamma_1(\Omega,F) \leq \frac{C}{\delta_0}, \quad \Gamma_2(\Omega,F) \leq C.$$
\end{lemma}

\begin{proof}
In order {to find such an $F$ and to estimate $\Gamma_1(\Omega,F)$, $\Gamma_2(\Omega,F)$,} we first construct an outward-pointing function.
Let $\varphi:\mathbb R \to \mathbb R$ be a $C^\infty$ function such that 
\begin{enumerate}[(i)]
	\item $\varphi(t)=t$ for $t\le \frac12$, 
	\item $\varphi(t)=c\in\left(\frac12,\frac34\right)$ for $t\ge\frac34$.
\end{enumerate}
We define 
\begin{equation*}
	g(x):=\varphi\left(\sigma(x)\frac{d(x,\partial \Omega)}{\delta_0}\right),
\end{equation*}
where
\begin{equation*}
	\sigma(x):=
	\begin{cases}
		1 & \mbox{ if }x\in \Omega,\\
		0 & \mbox{ if }x\in\partial \Omega,\\
		-1& \mbox{ if }x\in\mathbb R^n\setminus\overline{\Omega}.
	\end{cases}
\end{equation*}
If $x\in\Omega$ and $d(x,\partial \Omega)\ge \frac34 \delta_0$, $g(x)=c$. On the other-hand, it follows from the definition of $\delta_0$ that the function $x\mapsto d(x,\partial \Omega)$ is of class $C^2$ in the open inner tubular neighborhood of the boundary 
\[\partial \Omega^+_{\delta_0}:=\left\{x\in \Omega\,:\, d(x,\partial \Omega)<\delta_0 \right\}.\]
(See, for example, \cite[Lemma 14.16]{GT}).
It follows that $g$ is $C^2$ in $\Omega$, and in fact is $C^2$ in an open-neighborhood\footnote{The distance from $\partial \Omega$ up to which $g$ is $C^2$ in the complement of $\Omega$ will depend on $\Omega$, through the geometric parameter $\delta_1$. However, since we only consider the values of $g$ and its derivatives in $\overline{\Omega}$, {we just need $g$ to be $C^2$ in some open neighborhood of $\overline{\Omega}$}, even very small, and only the parameter $\delta_0$ will play a role.}  of $\overline{\Omega}$.

For $x$ in the inner tubular neighborhood $\partial \Omega^+_{\delta_0}$, {the chain rule implies}
\begin{equation}\label{eq:grad}
		\nabla g(x)=\frac1{\delta_0}\varphi'\left(\frac{d(x,\partial \Omega)}{\delta_0}\right) \nabla d(x,\partial \Omega).
		\end{equation} 
It is well-known that 
\begin{equation*}
	|\nabla d(x,\partial \Omega)|=1
\end{equation*}
and it is clear that there exists some absolute constant $C_1$ such that 
\begin{equation*}
	|\varphi'(t)|\le C_1
\end{equation*}
for all $t\in \mathbb R$. It follows that 
\begin{equation*}
	|\nabla g(x)|\le \frac{C_1}{\delta_0} 
\end{equation*}
for all $x\in \partial \Omega^+_{\delta_0}$. We note that Formula \eqref{eq:grad} extends by continuity to $x\in \partial \Omega$. Since $\varphi'(0)=1$, this implies, for all $x\in \partial \Omega$,
\begin{equation}\label{eq:normal}
	\nabla g(x)\cdot \nu(x) =\frac1{\delta_0}.
\end{equation}
In particular, $g$ is an outward-pointing function with constant $\gamma_g=\frac{1}{\delta_0}$. 

Furthermore, from Property (ii) of $\varphi$, $\nabla g(x)=0$ for all $x\in \Omega$ such that $d(x,\partial \Omega)\ge\frac34\delta_0$. It follows that 
\begin{equation}\label{eq:upper-grad}
	|\nabla g(x)|\le \frac{C_1}{\delta_0}
\end{equation}
for all $x\in\overline{\Omega}$.

Differentiating Formula \eqref{eq:grad} once more, we find, for $x\in \partial \Omega^+_{\delta_0}$, that 
\begin{align}\label{eq:laplacian}
	\Delta g(x)&=\frac{1}{\delta_0^2}\varphi''\left(\frac{d(x,\partial\Omega)}{\delta_0}\right)\left|\nabla d(x,\partial\Omega)\right|^2+\frac{1}{\delta_0}\varphi'\left(\frac{d(x,\partial\Omega)}{\delta_0}\right)\Delta d(x,\partial \Omega) \nonumber \\
	&=\frac{1}{\delta_0^2}\varphi''\left(\frac{d(x,\partial\Omega)}{\delta_0}\right)+\frac{1}{\delta_0}\varphi'\left(\frac{d(x,\partial\Omega)}{\delta_0}\right)\Delta d(x,\partial \Omega).
\end{align}
It is clear that there exists some absolute constant $C_2$ such that 
\[|\varphi''(t)|\le C_2\]
for all $t\in \mathbb R$. Property (ii) of $\varphi$ implies that $\Delta g(x)=0$ for all $x\in\Omega$ such that $d(x,\partial \Omega)\ge \frac{3}{4}\delta_0$. 
It remains to estimate $|\Delta d(x,\partial \Omega)|$.
Since $\partial \Omega$ is $C^2$, {we have, for all $x\in \partial \Omega_{\delta_0}^+$,}
\begin{equation*}
    \Delta d({x}, \partial \Omega) = \sum_{i=1}^{n-1} \frac{-\kappa_i}{1 - \kappa_i\, d({x}, \partial \Omega)},
\end{equation*}
{where the $\kappa_1,\dots,\kappa_{n-1}$ are the principal curvatures at $y$, the unique point in $\partial \Omega$ such that $|x-y|=d(x,\partial \Omega)$.}
(See, for example, \cite[Appendix 14.6]{GT}).
{Let us recall that, by definition, $\delta_0\le t_+$ and $\kappa_i\leq t_+^{-1}$ for all $1\le i\le n-1$. Then,
\begin{equation*}
    |\kappa_i\, d(x, \partial \Omega)| \leq |\kappa_i| \cdot \frac{3\delta_0}{4}\leq |\kappa_i| \cdot \frac{3t_+}{4} \leq \frac{3}{4}.
\end{equation*}
}
Therefore, if $x\in \partial\Omega^+_{3\delta_0/4}$,  
\begin{equation*}
	\left|\Delta d(x,\partial \Omega)\right| \leq \sum_{i=1}^{n-1} 4 |\kappa_i| \le \frac{4(n-1)}{t_+}\le \frac{4(n-1)}{\delta_0}.
\end{equation*}
This finally gives us
\begin{equation}\label{eq:upper-laplacian}
	\left|\Delta g(x)\right|\le \frac{C_2}{\delta_0^2}+\frac{4(n-1)C_1}{\delta_0^2}.
\end{equation}

Now we set
\[C:=C_2+4(n-1)C_1\]
(and we note that $C\ge 2 C_1$). Since $g$ is an outward-pointing function with constant $\gamma_g$, $F:=\nabla g$ is an outward-pointing vector field, with constant
\[\gamma_F=\gamma_g=\frac{1}{\delta_0}.\]
Using the formulas for $\Gamma_1(\Omega,F)$ and $\Gamma_2(\Omega,F)$ from Proposition \ref{prop:keyinequality}, we obtain
\begin{align*}
	\Gamma_1(\Omega,F)=&\frac{1}{\gamma_F}\sup_{\Omega}|{\rm div}(F)|=\delta_0\sup_{\Omega}|\Delta g|\le \delta_0\frac{C}{\delta_0^2}=\frac{C}{\delta_0},\\
	\Gamma_2(\Omega,F)=&\frac2{\gamma_F}\sup_\Omega|F|=2\,\delta_0\,\sup_\Omega|\nabla g|\le 2\,\delta_0\frac{C_1}{\delta_0}\le 2\,C_1\le C. 
\end{align*}
\end{proof}

In order to prove Theorem \ref{thm:eig}, we make use of the following standard inequalities.
\begin{rem}\label{rem:ineqs}
\begin{enumerate}[(i)]
Let $a$ and $b$ be non-negative numbers, $0\le\theta\le1$ and $p\ge1$.
    \item $(a+b)^\theta\le a^\theta+b^\theta$. If $\theta<1$ and $a,b>0$, the inequality is strict. A useful special case is $\theta=\frac12$: $\sqrt{a+b}\le \sqrt{a}+\sqrt{b}$. \label{eq:2}
    \item $ab^{p-1}\le C\left(a^p+b^p\right)$, with $C=C(p)$.\footnote{By this, we mean that for a given $p$, there exists a constant $C$, depending only on $p$, such that Inequality~\eqref{eq:3} is satisfied for all $a,b\ge0$. The meaning in the other statements in Remark \ref{rem:ineqs} is similar.} \label{eq:3}
      \item If $a\le b$, then $b^p-a^p\le p\,b^{p-1}(b-a)$. \label{eq:6}
    \item $(a+b)^p\le C\left(a^p+b^p\right)$, with $C=C(p)$. \label{eq:7}
    \item For all $\eps>0$, $(a+b)^p\le (1+\eps)\,a^p+C\,b^p$, with $C=C(p,\eps)$. \label{eq:8}
\end{enumerate}
\end{rem}
By using these inequalities in the proof of Theorem \ref{thm:eig}, we obtain the existence of a constant $C(n)$, depending only on $n$, as claimed. At each step of the proof, we take the maximum of various constants depending only on $n$ and thus any closed-form expression arising would be very complicated and far from optimal. Hence we use $C=C(n)$ throughout to denote a constant depending on $n$ which may change from line to line.

\begin{proof}[Proof of Theorem \ref{thm:eig}]
    We define
    \begin{equation}\label{eq:muHat}
    \hat\mu_H^{\frac12}:=\left(\mu+C\frac{H}{\delta_0}\right)^{\frac12}+CH,
    \end{equation}
    where $C$ is a constant depending only on $n$, defined by Lemma \ref{lem:bdsprop22}.
    The idea of the proof is to apply similar arguments to those used \cite{GL18} to $\hat\mu_H$ instead of to $\mu$.
    
    We assume $\mu\ge0$. Then, according to Remark \ref{rem:ineqs} part \eqref{eq:2},
\[\hat\mu_H^{\frac12}\le \mu^{\frac12}+\Delta_H,\]
with
\[\Delta_H:=\left(C\frac{H}{\delta_0}\right)^{\frac12}+CH.\]

We fix $\hat\eps(n)\in (0,1)$ (to be specified later). We consider a spectral pair (eigenvalue-eigenfunction) $(\mu,u)$ for the Robin eigenvalue problem. As in \cite{cL16, GL18, HS}, we consider $\nu_0$ and $\nu_1$, respectively the number of bulk and boundary nodal domains (defined using $\eps_0=\hat\eps(n)$), and $\hat\mu_H$, associated with $\mu$ via Equation \eqref{eq:muHat}.

If 
\begin{equation}\label{eq:cond-above}
\left(\frac{V^{\frac2n}}{\hat\mu_H}\right)^{\frac14}\le \delta_1,
\end{equation}
that is, if 
\begin{equation}\label{eq:cond1}
\hat\mu_H\ge \frac{V^{\frac2n}}{\delta_1^4},
\end{equation}
then, by following the same arguments as  in \cite{cL16,GL18}, we have
\begin{equation}\label{eq:nu0}
	\nu_0\le\frac{V}{\Lambda(n)^{\frac{n}{2}}}\left((1+\eps(n))\hat\mu_H+C\,\frac{\hat\mu_H^{\frac12}}{V^{\frac1n}}\right)^{\frac{n}{2}},
\end{equation}
where $\Lambda(n)$ equals $\lambda_1(B)$, the first Dirichlet eigenvalue of the ball $B\subset \R^n$ of volume $1$, and $\eps(n)>0$ is defined by
\[1+\eps(n)=\frac{1+\hat\eps(n)}{1-\hat\eps(n)}\]
and can be chosen arbitrarily small by taking $\hat\eps(n)$ small enough, at the cost of a larger $C$. In what follows, we express the inequalities in terms of $\eps(n)$, for which we will specify some properties later (this correspondingly specifies $\hat\eps(n)$). Let us note the crucial role played by Proposition \ref{prop:keyinequality} at this stage of the proof: it gives us an upper bound for the Rayleigh quotient of the eigenfunction $u$ in one of its nodal domains, which allows us to extend the analysis of \cite{cL16,GL18} to the Robin problem considered here, which has a possibly negative $h$.

Still under Condition \eqref{eq:cond1}, we have\footnote{Inequalities \eqref{eq:nu0} and \eqref{eq:nu1} essentially correspond to the upper bounds on page {123} of \cite{GL18}. However, there is a gap in the proof of the upper bound for $\nu_1$ in that reference. It can be fixed at the cost of changing the values of the constants and putting a stronger condition on $\mu$. Namely, one must impose that Inequality \eqref{eq:cond-above} is satisfied, rather than Inequality (34) in \cite{GL18}, meaning that $\delta_0$ must be replaced with $\delta_1$. That way, one obtains \eqref{eq:nu1}.}
\begin{equation}\label{eq:nu1}
	\nu_1\le C(n)\,S\,V^{\frac{1}{2n}}\hat\mu_H^{-\frac14}\left(\hat\mu_H+\frac{\hat\mu_H^{\frac12}}{V^{\frac1n}}\right)^{\frac{n}{2}}.
\end{equation}

Now, as in \cite[Section 2.1]{GL18}, let $(\lambda_k(\Omega))_{k\ge1}$ denote the Dirichlet eigenvalues of the Laplacian on $\Omega$ and, for $\mu>0$, define the Dirichlet counting function as:
\begin{equation*}
	N_\Omega^{D}(\mu):=\sharp\{k\in\N\,:\,\lambda_k(\Omega)<\mu\},
\end{equation*}
and the corresponding remainder $R_{\Omega}^{D}(\mu)$ such that
\begin{equation}
\label{eqDefRemainder}
	N_{\Omega}^{D}(\mu)=\frac{\omega_n \vert \Omega \vert}{(2\pi)^{n}}\mu^{n/2}-R_{\Omega}^{D}(\mu),
\end{equation}
where the first term in the right-hand side of Equation \eqref{eqDefRemainder} corresponds to Weyl's law.

By monotonicity of the Robin eigenvalues, we have
\begin{equation*}
	N_{\Omega}^h(\mu)\ge N_{\Omega}^D(\mu)=w_n\,V\,\mu^{\frac{n}{2}}-R_{\Omega}^D(\mu)
\end{equation*}
with $w_n:=\omega_n/(2\pi)^n$. From the known bounds on the Dirichlet counting function (see \cite[Section 9.2]{GL18} and \cite[Section 2]{vdBkG16cs}), we {find} that under a condition of the form
\begin{equation}\label{eq:cond2}
\mu \ge C\frac{V^2}{S^2\,\delta_0^4},
\end{equation}
we have
\begin{equation*}
	R_{\Omega}^D(\mu)\le C\,(SV)^{\frac12}\mu^{\frac{n}{2}-\frac14} \le C\,(SV)^{\frac12}\left(\hat\mu_H\right)^{\frac{n}{2}-\frac14}.
\end{equation*}
It follows that 
\begin{equation*}
	N_{\Omega}^h(\mu)\ge w_n\, V\, \hat\mu_H^{\frac{n}{2}}-C\,(SV)^{\frac12}\hat\mu_H^{\frac{n}{2}-\frac14}-w_n\,V\,\left(\hat\mu_H^{\frac{n}{2}}-\mu^{\frac{n}{2}}\right).
\end{equation*}
Using Remark \ref{rem:ineqs} part \eqref{eq:6} with $a=\mu^{\frac12}$, $b=\hat\mu_H^{\frac12}$ and $p=n$, we get 
\begin{equation*}
	\hat\mu_H^{\frac{n}{2}}-\mu^{\frac{n}{2}}\le n\,\hat\mu_H^{\frac{n}{2}-\frac12}\left(\hat\mu_H^{\frac12}-\mu^{\frac12}\right)\le n\,\hat\mu_H^{\frac{n}{2}-\frac12}\,\Delta_H.
\end{equation*}
Hence we obtain
\begin{equation} \label{eq:counting}
	N_{\Omega}^h(\mu)\ge w_n\, V\, \hat\mu_H^{\frac{n}{2}}-C\,(SV)^{\frac12}\hat\mu_H^{\frac{n}{2}-\frac14}-C\,V\,\hat\mu_H^{\frac{n}{2}-\frac12}\,\Delta_H.
\end{equation}

A necessary condition for $(\mu,u)$ to be a Courant-sharp pair is 
\[N_\Omega^h(\mu)-\nu_0-\nu_1<0.\]
If $\mu\ge0$ and Conditions \eqref{eq:cond1} and \eqref{eq:cond2} are satisfied, we have
\begin{align}
	N_\Omega^h(\mu)-\nu_0-\nu_1&\ge w_n\, V\, \hat\mu_H^{\frac{n}{2}}-C\,(SV)^{\frac12}\hat\mu_H^{\frac{n}{2}-\frac14}-C\,V\,\hat\mu_H^{\frac{n}{2}-\frac12}\,\Delta_H \nonumber\\
	&-\frac{V}{\Lambda(n)^{\frac{n}{2}}}\left((1+\eps(n))\hat\mu_H+C\,\frac{\hat\mu_H^{\frac12}}{V^{\frac1n}}\right)^{\frac{n}{2}} \nonumber\\
	&-C\,S\,V^{\frac{1}{2n}}\hat\mu_H^{-\frac14}\left(\hat\mu_H+\frac{\hat\mu_H^{\frac12}}{V^{\frac1n}}\right)^{\frac{n}{2}}. \label{eq:lbdcount}
\end{align} 

We choose to express the various inequalities in terms of the rescaled quantities
\begin{align*}
\xi&:=V^{\frac1{n}}\mu^{\frac{1}{2}},\\
\hat\xi&:=V^{\frac1{n}}\hat\mu_{H}^{\frac{1}{2}}.
\end{align*}
Inequality \eqref{eq:lbdcount} then becomes
\begin{align*}
	N_\Omega^h(\mu)-\nu_0-\nu_1\ge&w_n\, \hat\xi^n-C\,\rho^{\frac12}\hat\xi^{n-\frac12}-C\,V^{\frac1n}\,\Delta_H\,\hat\xi^{n-1}\\
	&-\frac{1}{\Lambda(n)^{\frac{n}{2}}}\left((1+\eps(n))\hat\xi^2+C\,\hat\xi\right)^{\frac{n}{2}}\\
	&-C\,\rho\,\hat\xi^{-\frac12}\left(\hat\xi^2+\hat\xi\right)^{\frac{n}{2}}.
\end{align*} 
If we apply Remark \ref{rem:ineqs} part \eqref{eq:8} to the term on the second line (with $p=\frac{n}2$ and $\eps=\eps(n)$) and rearrange the right-hand side, we find
\begin{align}
	N_\Omega^h(\mu)-\nu_0-\nu_1\ge&\left(w_n-\frac{(1+\eps(n))^{\frac{n}{2}+1}}{\Lambda(n)^{\frac{n}{2}}}\right)\, \hat\xi^n-C\,\rho^{\frac12}\hat\xi^{n-\frac12}-C\,V^{\frac1n}\,\Delta_H\,\hat\xi^{n-1} \nonumber\\
	&-C\,\hat\xi^{\frac{n}{2}}-C\,\rho\,\hat\xi^{-\frac12}\left(\hat\xi^2+\hat\xi\right)^{\frac{n}{2}}.
 \label{eq:lbdcount2}
\end{align} 

Since
\[w_n>\frac{1}{\Lambda(n)^{\frac{n}{2}}},\]
we can choose $\eps(n)>0$ such that 
\[\ell(n):=w_n-\frac{(1+\eps(n))^{\frac{n}{2}+1}}{\Lambda(n)^{\frac{n}{2}}}>0.\]
Furthermore, we have, from the isoperimetric inequality, $\rho\ge n\omega_n^{\frac{1}{n}}$. It follows that for any $0\le \alpha<\beta$, 
\[\rho^{\alpha}\le C\rho^{\beta},\]
with $C=C(n,\alpha,\beta):=(n\omega_n^{1/n})^{\alpha-\beta}$. In addition, if $\hat\xi\ge\rho$, we have, with the same $C$,
\[\hat\xi^{\alpha}\le C\hat\xi^{\beta}.\]
Using these remarks in Inequality \eqref{eq:lbdcount2}, we obtain, if $\hat\xi\ge \rho$, 
\begin{equation*}
	N_\Omega(\mu)-\nu_0-\nu_1\ge \ell(n)\,\hat\xi^n-C\left(\rho+V^{\frac{1}{n}}\Delta_H\right)\,\hat\xi^{n-\frac12}.
\end{equation*} 
Now, using Remark \ref{rem:ineqs} part \eqref{eq:3} with $a=\delta_0^{-\frac12}$, $b=H^{\frac12}$ and $p=2$, we find
\[\Delta_H\le C\left(\frac1{\delta_0}+H\right).\]
Therefore, if $\hat\xi\ge \rho$, 
\begin{equation}\label{eq:lower}
	N_\Omega(\mu)-\nu_0-\nu_1\ge\ell(n)\,\hat\xi^n-C\left(\rho+\frac{V^{\frac1n}}{\delta_0}+V^{\frac1n}H\right)\,\hat\xi^{n-\frac12}.
\end{equation}
If $\hat\xi\ge\hat\xi^*$, with 
\[\hat\xi^*:=\frac{C^2}{\ell(n)^2}\left(\rho+\frac{V^{\frac1n}}{\delta_0}+V^{\frac{1}{n}}H\right)^2,\]
then the right-hand side of Inequality \eqref{eq:lower} is non-negative. Note that by applying Remark \ref{rem:ineqs} part \eqref{eq:7} with $p=2$ repeatedly, with find that there exists $C$ large enough so that 
\[\hat\xi^*\le C\left(\rho^2+\frac{V^{\frac2n}}{\delta_0^2}+V^{\frac{2}{n}}H^2\right).\]

Putting everything together, we conclude that there exists a constant $C$ (large enough) so that under conditions \eqref{eq:cond1} and \eqref{eq:cond2}, if
\begin{equation*}
\hat\xi\ge C\left(\rho^2+\frac{V^{\frac2n}}{\delta_0^2}+V^{\frac{2}{n}}H^2\right),
\end{equation*}
then $(u,\mu)$ cannot be Courant-sharp (we note that $C$ can be chosen large enough that the above inequality implies $\hat\xi\ge \rho$). 

Conditions \eqref{eq:cond1} and \eqref{eq:cond2} are not independent. Since $\delta_1\le\delta_0$ by definition and $\rho\ge n\omega_n^{1/n}$, we have
\[\frac{V^2}{S^2\delta_0^4}\le\frac{V^{\frac{2}{n}}}{\rho^2\delta_1^4}\le C\frac{V^{\frac{2}{n}}}{\delta_1^4},\]
with $C=(n\omega_n)^{-2}$. Under the condition $\mu\ge0$, we have $\hat\mu_H\ge\mu$. Therefore, we can choose $C$ (large enough) so that 
\[\mu\ge C\frac{V^{\frac{2}{n}}}{\delta_1^4}\] 
implies both conditions \eqref{eq:cond1} and \eqref{eq:cond2}.

It follows that, if $(\mu,u)$ is Courant-sharp,
\[\xi\le \max\left\{C\frac{V^{\frac{2}{n}}}{\delta_1^2}, C\left(\rho^2+\frac{V^{\frac2n}}{\delta_0^2}+V^{\frac{2}{n}}H^2\right)\right\},\]
which, {since} $\delta_1\le \delta_0$, implies
\begin{equation}
\label{eq:xi-upper}
\xi\le C \left(\frac{V^{\frac2n}}{\delta_1^2}+\rho^2+V^{\frac{2}{n}}H^2\right).
\end{equation}

From $\mu=\xi^2/V^{\frac2n}$ and from Remark \ref{rem:ineqs} part \eqref{eq:7}, applied repeatedly with $p=2$, we obtain that, for $(\mu,u)$ Courant-sharp,
 \[\mu\le C \left(\frac{V^{\frac2n}}{\delta_1^4}+\frac{\rho^4}{V^{\frac2n}}+V^{\frac{2}{n}}H^4\right).\]
\end{proof}

\subsection{Geometric upper bound for the number of Courant-sharp Robin eigenvalues of a convex, $C^2$ domain}
In this section we take $\Omega \subset \R^n$, $n \geq 2$, to be an open, bounded, connected, convex set with $C^2$ boundary. We use the same notation as was introduced at the beginning of Section \ref{ss:geombdeval}. We note that the convexity assumption on $\Omega$ allows us to employ Corollary \ref{corCountUpper}.

Our main result is the following.
\begin{thm} \label{thm:number} 
Let $\Omega \subset \R^n$, $n \geq 2$, be an open, bounded, connected, convex set with $C^2$ boundary.
There exists a constant $C$, depending only on $n$, such that the number of Courant-sharp Robin eigenvalues of $\Omega$ is at most
\[{C}\left(\frac{V^{2}}{t_+^{2n}}+\rho^{2n}+V^{2}H^{2n}\right).\]
\end{thm}
{We note that in Theorem \ref{thm:number} $h$ can be negative. A similar result for the case where $h \geq 0$ was obtained in \cite[Propositions 8.3 and 9.4]{GL18}. }

Roughly speaking, the strategy of the proof is to substitute the result of Theorem \ref{thm:eig} into that of Corollary \ref{corCountUpper}. As the constant given in the statement of Theorem \ref{thm:eig} may be very large, the constant $C$ given in the statement of Theorem \ref{thm:number} may be very large too.

\begin{proof}
By Lemma \ref{lem:bdsprop22}, we see that there exists a constant $C>0$, depending only on $n$ {and a vector field $F:\Omega\to\mathbb R^n$, outward-pointing relative to $\Omega$,} such that Proposition \ref{propComparisonNR} holds with
$$ K_1(\Omega,F) \leq \frac{C}{{\delta_0}}, \quad K_2(\Omega,F) \leq C.$$
{Thus}, there exists a constant $K$, {depending only on $n$,} such that, for all $\eta\in(0,1)$ and all $x\ge0$,
\begin{equation}
	N_{\Omega}^h(x)\le N_{\Omega}^N\left(\frac1{1-\eta}\left(x+K\left(\frac{H}{\delta_0}+\frac{H^2}{\eta}\right)\right)\right). \label{eq:Countupper4}
\end{equation}
Since $\Omega$ is convex, inequality \eqref{eq:CountUpper2} gives an upper bound for the Neumann counting function, which we can put into a simplified form: there exists a constant $C$, depending only on $n$, such that for all $x\geq 0$,
\begin{equation}
	N_{\Omega}^N(x)\le C\,x^{\frac{n}{2}}\left(V+S\,x^{-\frac{1}{2}}\left(1+\frac{1}{t_+\sqrt{x}}\right)^{n-1}\right). \label{eq:Countupper3}
\end{equation}

We begin by deducing a weaker inequality from Inequality \eqref{eq:Countupper3} which is easier to use. Setting $y=V^{\frac1n}x^{\frac12}$ (dimensionless variable), we have
\begin{equation*}
N_\Omega^N\left(x\right)\le C\,y^n\left(1+\rho\,y^{-1}\left(1+\frac{V^{\frac1n}}{t_+y}\right)^{n-1}\right)
							  = C\left(y^n+\rho\,\left(y+\frac{V^{\frac1n}}{t_+}\right)^{n-1}\right).
\end{equation*}

Using Remark \ref{rem:ineqs} part \eqref{eq:7} and part \eqref{eq:3} (twice), we get
\begin{equation*}
N_\Omega^N\left(x\right)\le C\left(y^n+\rho\,y^{n-1}+\rho\,\left(\frac{V^{\frac1n}}{t_+}\right)^{n-1}\right)
							  \le C\left(y^n+\rho^n+y^n+\rho^n+\left(\frac{V^{\frac1n}}{t_+}\right)^{n}\right).		
\end{equation*}
We obtain, for $x\ge0$,
\begin{equation}\label{eq:upper-counting}
	N_\Omega^N\left(x\right) \le C\left(y^n+\rho^n+\left(\frac{V^{\frac1n}}{t_+}\right)^{n}\right).
\end{equation}

From Inequality \eqref{eq:Countupper4} (fixing, for instance, $\eta=\frac12$), and Inequality \eqref{eq:upper-counting}, we get, for all $x$, 
\begin{equation*}
	N_\Omega^h\left(x\right) \le C\left(V\left(C\left(x_++\frac{H}{t_+}+H^2\right)\right)^{\frac{n}{2}}+\rho^n+\left(\frac{V^{\frac1n}}{t_+}\right)^{n}\right),
\end{equation*}
where $x_+=\max\{0,x\}$\footnote{We give this definition to deal with the case where $x<0$, which is relevant since the Robin Laplacian has negative eigenvalues in general. }. Setting $y_{+}:=V^{\frac1n}x_+^{\frac12}$, using repeatedly Remark \ref{rem:ineqs} part \eqref{eq:7} with $p=\frac{n}{2}$ and part \eqref{eq:3} with $p=2$, we get
\begin{align*}
	N_\Omega^h\left(x\right) &\le C\left(y_{+}^n+ \left(\frac{V^{\frac2n}H}{t_+}\right)^{\frac{n}2}+(V^{\frac1n}H)^n+\rho^n+\left(\frac{V^{\frac1n}}{t_+}\right)^{n}\right)\\
&\le 	C\left(y_{+}^n+ \left(\left(\frac{V^{\frac1n}}{t_+}\right)^2+(V^{\frac1n}H)^2\right)^{\frac{n}2}+(V^{\frac1n}H)^n+\rho^n+\left(\frac{V^{\frac1n}}{t_+}\right)^{n}\right)\\
&\le 	C\left(y_{+}^n+ \left(\frac{V^{\frac1n}}{t_+}\right)^n+(V^{\frac1n}H)^n+(V^{\frac1n}H)^n+\rho^n+\left(\frac{V^{\frac1n}}{t_+}\right)^{n}\right).
\end{align*}
Finally, we obtain
\begin{equation}\label{eq:robin-upper}
	N_\Omega^h\left(x\right)\le C\left(y_{+}^n+(V^{\frac1n}H)^n+\rho^n+\left(\frac{V^{\frac1n}}{t_+}\right)^{n}\right).
\end{equation}

We now denote the largest Courant-sharp Robin eigenvalue by $\mu=\mu_k(\Omega,h)$. Then, the number of Courant-sharp Robin eigenvalues (counted with multiplicities) is at most $k$. For $\eps>0$ (using the notation $\mu_+=\max\{\mu,0\}$ and $\xi=V^{\frac1n}\mu_+^{\frac12}$), we have
\[k\le N_\Omega^h(\mu_++\eps)\le C\left((\mu_++\eps)^{\frac{n}{2}}V+(V^{\frac1n}H)^n+\rho^n+\left(\frac{V^{\frac1n}}{t_+}\right)^{n}\right),\] 
so that, taking $\eps\to 0^+$, we get
\begin{equation}
k\le C \left(\xi^n+(V^{\frac1n}H)^n+\rho^n+\left(\frac{V^{\frac1n}}{t_+}\right)^{n}\right). \label{eq:Countupper5}
\end{equation}

We now use Theorem \ref{thm:eig}, with the conclusion written as inequality \eqref{eq:xi-upper}, to obtain
 \begin{equation}
 	\xi\le C \left(\frac{V^{\frac2n}}{t_+^2}+\rho^2+V^{\frac2n}H^2\right), \label{eq:xi-upper2}
 \end{equation}
 where we used the fact that $\delta_1=\delta_0=t_+$ since $\Omega$ is convex.
 Substituting Inequality \eqref{eq:xi-upper2} for $\xi$ in Inequality \eqref{eq:Countupper5} and using Remark \ref{rem:ineqs} part \eqref{eq:7} with $p=n$ repeatedly, we get
 \begin{equation}
 k\le C \left(\left(\frac{V^{\frac1n}}{t_+}\right)^{2n}+\rho^{2n}+(V^{\frac1n}H)^{2n}+(V^{\frac1n}H)^{n}+\rho^n+\left(\frac{V^{\frac1n}}{t_+}\right)^{n}\right). \label{eq:Countupper6}
 \end{equation}
To complete the proof, we note that $\rho\ge n\omega_n^{\frac1n}$ (from the isoperimetric inequality), so that 

\begin{align*}
	\rho^n&\le C \rho^{2n},\\
	(V^{\frac1n}H)^n &\le C\,\rho^n (V^{\frac1n}H)^{n} \le C \left(\rho^{2n}+(V^{\frac1n}H)^{2n}\right),\\
	\left(\frac{V^{\frac1n}}{t_+}\right)^{n} &\le C \rho^n \left(\frac{V^{\frac1n}}{t_+}\right)^{n} \le C \left(\rho^{2n}+\left(\frac{V^{\frac1n}}{t_+}\right)^{2n}\right).
\end{align*} 
Substituting these bounds for the $n$-th powers in Inequality \eqref{eq:Countupper6}, we conclude that
  \[k\le C \left(\left(\frac{V^{\frac1n}}{t_+}\right)^{2n}+\rho^{2n}+(V^{\frac1n}H)^{2n}\right).\]

\end{proof}

\begin{rem}
    We note that the bounds derived in the proof of Theorem \ref{thm:eig} implicitly contain a bound on the number of negative Robin eigenvalues of $\Omega$. Indeed, for all $x<0$, $x_+=0$ and therefore $y_{+}=0$. This means that for negative values of $x$, taking $y_{+}=0$ in the right-hand side of \eqref{eq:robin-upper} gives an upper bound for the number of negative Robin eigenvalues, $N^h_{\Omega}(x)$, which is explicit in terms of $H$ and some of the geometric quantities of $\Omega$.
\end{rem}

\appendix
\section{Proof of Proposition \ref{prop:fkq}}
The idea behind the quantitative version of the Faber-Krahn inequality for the first Dirichlet eigenvalue (see \cite{FMP09} and the references therein) employs the classical method of comparing the Dirichlet energy of the eigenfunction with that of its radially symmetric, monotone rearrangement on a Euclidean ball, known as the Pólya-Szegö inequality, together with a quantitative version of the Euclidean isoperimetric inequality. The proof of Proposition \ref{prop:fkq} follows a similar line of reasoning. We first establish a quantitative version of the isoperimetric inequality in terms of the perimeter of the interior boundary of `small' domains in  ${\Omega}$ and a modified Fraenkel asymmetry (see Proposition \ref{prop:isoperimetricinq}). We follow the proof of the (non-quantitative) isoperimetric inequality in \cite{DFV}, incorporating the quantitative isoperimetric inequality and applying it to part of the domain distant from the boundary. It in turn gives an improvement in the Pólya-Szegö inequality. Then we follow the proof in \cite{DFV} while adapting the approach in \cite{FMP09} to obtain a quantitative version of the Faber-Krahn inequality in the setting of Proposition \ref{prop:fkq}. While the method of the proof is based on the same ideas as in \cite{FMP09,DFV},  the presence of a modified Fraenkel asymmetry and a modification of the eigenfunction in the Pólya-Szegö Inequality pose some challenges and require adaptation at each step of the proof.  \\

We first start with some definitions and notation. Throughout this section, we use $C_i=C_i(n)$ to denote constants depending only on $n$, and for the rest of the constants we mainly use  $c_i=c_i(\cdot)$.\\

For any $f\in L^1_{\loc}(\R^n)$ and $U\subset \R^n$ open, we define
\[|\bD f|(U):=\inf\left\{\liminf_{n\to\infty}\int_U|\nabla f_n|: f_n\in\Lip_{\loc}(U),~ \text{$f_n\to f$ in $L^1_{\loc}(U)$}  \right\}.\]
For any Borel set $A \subseteq {\R^n}$, we set
 \[|\bD f|(A):=\inf\left\{|\bD f|(U): A\subset U,\text{~ and $U$ is open in $\R^n$}\right\}.\] Subsequently for any two Borel sets $A$ and $E$ we define $\per(E,A)=|\bD \chi_E|(A)$. We denote $\per(E,\R^n)$ by $\per(E)$ or $|\partial E|$.  We say that $E$ is of finite perimeter when $\per(E)<\infty$.

 Let $\Omega$ be an open Lipschitz domain in $\mathbb{R}^n$. 
 Note that $\per(E,\Omega)$ gives the interior perimeter of $E$ in $\Omega$.  
 For simplicity, we use the following notation:
 \[\per(E,\Omega)=|\partial E\cap \Omega|.\]
 
De Ponti, Farinelli, and Violo \cite[Theorem 4.1]{DFV} proved an almost sharp isoperimetric inequality for any Borel set $E$ in ${\Omega}$ with `\textit{small}' volume having its interior boundary $\per(E,\Omega)$  instead of $\per(E)$ in the inequality. We shall see that with some modification of their proof, we get a quantitative version.  It is an interesting question whether the quantitative version we obtain for the Lipschitz domains can be generalised to the general setting of the PI spaces considered in \cite{DFV}.

\begin{prop}\label{prop:isoperimetricinq}
    Let {$\Omega \subset \R^n$} be an open, bounded, Lipschitz domain. For any $\epsilon,\delta\in(0,1)$  there exist  a neighbourhood $\Omega_\delta$ of $\partial \Omega$ with $|\Omega_\delta|<\delta$ and positive constants $C_1=C_1(n)$, $\alpha=\alpha(\Omega, n,\epsilon,\delta)$ and $\beta=\beta(\epsilon)$ such that  for every Borel set $E\subset \Omega$ with

\begin{equation}\label{eq:smallv}
0 < |E| \leq \alpha, \qquad \frac{|E \cap \Omega_\delta|}{|E|} \leq \beta,
\end{equation}
 we have
\begin{equation}
|\partial E \cap \Omega| \geq (1-\epsilon) \left(1 + C_1\tilde A(E)^2 \right) n\omega_n^{1/n} |E|^{\frac{n-1}{n}}
\end{equation}
where $\tilde A(E):=\inf_U A(E\cap U)$, where the infimum is taken over all open sets $U\subseteq\Omega$ such that $\Omega\setminus\Omega_\delta\subset U$.

\end{prop}

\begin{proof}Let $U_t=\{x\in\bOm: d(x,\partial\Omega)>t\}$. Take $t_0>0$ such that $\Omega_\delta:=\Omega\setminus U_{t_0}$ has volume less than $\delta$. We first show that for any Borel set $E\subset\bOm$, there exists $t_1\in(0,t_0)$ such that
\[|\partial (E\cap U_{t_1})|\le |\partial E\cap \Omega|+c_1|E|,\]
where $c_1=c_1(\Omega,\delta)$ is a positive constant.

Note that the function $d(\cdot,\partial\Omega)\in \Lip(\bOm)$. By the co-area formula \cite[Theorem 4.2]{Mir}, we have
\[\int_{t_0/2}^{t_0}\per(U_t,E)dt\le\int_{0}^\infty\per(U_t,E)dt=\int_E|\nabla d|\le|E|.\] 
By Markov's inequality, 
\begin{align*}\left|\left\{t\in (\frac{t_0}{2},t_0): \per(U_t,E)\ge\frac{3|E|}{t_0} \right\}\right|&\le\frac{t_0}{3|E|}\int_{t_0/2}^{t_0}\per(U_t,E)dt\le\frac{t_0}{3}.\end{align*}
Hence, the set 
$$T:=\left\{t\in (\frac{t_0}{2},t_0): \per(U_t,E)<\frac{3|E|}{t_0} \right\}$$
 has non-zero measure. Therefore, there exists $t_1\in T$ such that $\per(E,\partial U_{t_1})=0$  following the same argument as in \cite[Corollary 2.6]{Ant}. Here, $t_0$ and $t_1$ depend only on $\Omega$ and $\delta$. Moreover, since  $\per(E,\partial U_{t_1})=0$, we can apply  \cite[Proposition 2.6]{Ant2} to get 
\begin{align*}|\partial(E\cap U_{t_1})|&=\per(E\cap U_{t_1})\le \per(E,U_{t_1})+\per(U_{t_1},E^\circ)\\
&\le \per(E,\Omega)+\frac{3|E|}{t_0},
\end{align*}
where $E^\circ$ denotes the interior of $E$.\\
Since $E\cap U_{t_1}\subset\Omega$, we can use the quantitative version of the isoperimetric inequality \cite{FMP08}:
\[|\partial(E\cap U_{t_1})|\ge (1+{C_1}A(E\cap U_{t_1})^2)n\omega_n^{1/n}|E\cap U_{t_1}|^{\frac{n-1}{n}}.\]

Combining the above inequalities, we obtain
\begin{align*}
    |\partial E\cap \Omega|=\per(E,\Omega)&\ge(1+C_1A(E\cap U_{t_1})^2)n\omega_n^{1/n}|E\cap U_{t_1}|^{\frac{n-1}{n}}-\frac{3|E|}{t_0}\\
    &\ge(1+C_1A(E\cap U_{t_1})^2)n\omega_n^{1/n}|E\setminus \Omega_\delta|^{\frac{n-1}{n}}-\frac{3|E|}{t_0}\\
    &\ge(1+C_1A(E\cap U_{t_1})^2)n\omega_n^{1/n}|E|^{\frac{n-1}{n}}\left(1-\left(\frac{|E\cap\Omega_\delta|}{|E|}\right)^{\frac{n-1}{n}}-\frac{3|E|^{1/n}}{t_0 n\omega_n^{1/n}}\right).
\end{align*}
Hence, it is enough to have 
$$\left(\frac{|E\cap\Omega_\delta|}{|E|}\right)^{\frac{n-1}{n}}\le \frac{\epsilon}{2},\qquad\frac{3|E|^{1/n}}{t_0 n\omega_n^{1/n}}\le\frac{\epsilon}{2}.$$
We conclude by taking $\beta=\left(\frac{\epsilon}{2}\right)^{\frac{n}{n-1}}
$ and  $\alpha=\frac{t_0^n}{6^n}n^n\omega_n\epsilon ^{n}$.

\end{proof}

Another main ingredient of the proof is a generalization of the P\'olya-Szeg\"o inequality.
The classical  P\'olya-Szeg\"o inequality is stated for functions in $H^1_0(U)$ where $U$ is a bounded open set in $\R^n$. Recently the P\'olya-Szeg\"o inequality has been extended to the setting of metric-measure spaces \cite{MS20,NV22,DFV}. Its main difference with the classical version is that we can view $\bOm$ as a complete metric-measure space and compactly supported functions in $\bOm$ can be nonzero on the boundary of $\Omega$.  Without any extra condition, the P\'olya-Szeg\"o inequality fails in this general setting.  However, it holds true on subdomains of $\bOm$ with small volume. In \cite[Theorem 2.22]{DFV}, it is stated that {it holds} for open sets in a bounded \textit{PI  space}. We do not need to know the definition of a PI space. We use \cite[Theorem 3.9]{DFV}, where it is shown that a bounded Lipschitz domain $\bOm$ is a bounded PI space. Hence, we state Lemma~\ref{lemma:pz} below, which is a version of \cite[Theorem 2.22]{DFV}, only for Lipschitz domains. Let us first introduce some notation. 

Let $\Omega$ be an open, bounded, Lipschitz domain in $\R^n$. Let $U\subsetneq\bOm$ be an open set in $\bOm$. For a non-negative Borel function $u$ on $U$ we define 
$$\mu(t)=|\{u>t\}|.$$ 
We denote by $U^*$ the ball centred at the origin having the same volume as $U$, and $u^*:U^*\to[0,\infty)$ the symmetric decreasing rearrangement of $u$. By abuse of notation, we use $\Lip_c(U)$ to denote the space of Lipschitz functions compactly supported in $U$ with induced topology from $\bOm$.
\begin{lemma}[P\'olya-Szeg\"o inequality]\label{lemma:pz}Let $\Omega$ be an open, bounded, Lipschitz domain in $\R^n$. There exists a constant $c=c(\Omega)$ such that for any open set $U\subsetneq\bOm$ in $\bOm$ with $|U|\le c$, 
any  $0\neq u\in \Lip_c(U)$ non-negative with $|\nabla u|\neq0$ a.e. in $\{u>0\}$, and any $0<s\le T:=\sup_{U} u$ we have $u^*\in \Lip_c(U^*)$ and
\begin{equation*}
     \int_{\{u\le s\}}|\nabla  u|^2- \int_{\{u^*\le s\}}|\nabla  u^*|^2\ge \int_{0}^{s}\frac{\per(\{u>t\},\Omega)^2-\per(\{u^*>t\})^2}{|\mu'(t)|}.
\end{equation*}

\end{lemma}

\begin{proof} By \cite[Lemma 3.7]{DFV}, $\bOm$ satisfies the assumption of \cite[Theorem 2.18]{DFV}. Thus there exist constants $c_0=c_0(\Omega)>0$ and $c_1=c_1(\Omega,n)$ such that for any $U$ with $|U|\le c_0$ and any Borel set $E\subset U$ we have
\[\per(E, {\Omega})=|\partial E\cap \Omega|\ge c_1|E|^{\frac{n-1}{n}}.\] 
Hence, the assumption of \cite[Theorem 2.22]{DFV} is met. Therefore, $u^*\in \Lip_c(U^*)$. Moreover, for  any $s\in (0,T)$, we have (see e.g. \cite[Lemma 2.25]{NV22}) $$-\mu'(t)=\int_{\{u^*=t\}}\frac{d\cH^{n-1}}{|\nabla u^*|}= \int_{\{u=t\}}\frac{d\cH^{n-1}}{|\nabla u|},\quad a.e.$$ and as a result (see \cite[Inequality (2.29)]{DFV})

\begin{align}\label{eq:pz1inq}
 \int_{\{u\le s\}}|\nabla  u|^2   & \ge\int_{0}^{s }\frac{\per(\{u>t\},\Omega)^2}{-\mu'(t)}dt.
   \end{align}
Since $|\nabla u^*|$ is constant on $\{u^*=t\}$, we have  
  \begin{align}\label{eq:pz2inq}
  \int_{\{u^*\le s\}}|\nabla  u^*|^2
   &=\int_{0}^{s}dt\int_{\{u^*=t\}}|\nabla u^*|{d{\cH}^{n-1}} =\int_{0}^{s}\frac{\per(\{u^*>t\})^2}{-\mu'(t)}dt.
   \end{align}     
We obtain the result by taking the difference of \eqref{eq:pz1inq} and \eqref{eq:pz2inq}.
\end{proof}
We now use Proposition \ref{prop:isoperimetricinq}, and Lemma \ref{lemma:pz} to prove Proposition \ref{prop:fkq}. 
\begin{proof}[Proof of Proposition \ref{prop:fkq}]
    We choose $\vtheta_0=\vtheta_0(\Omega, \epsilon,\delta),\vtheta_1=\vtheta_1(\Omega, \epsilon)\in(0,1)$ small enough such that $\vtheta_0$ satisfies inequality \eqref{eq:delta0-condition1}  below, and $\vtheta_1$ satisfies a  set of inequalities \eqref{eq:delta_1condition 1},\eqref{eq:delta1-condition2},\eqref{eq:delta_1-condition3},\eqref{eq:delta1-condition4},\eqref{eq:delta1-condition5},\eqref{cond:delta1} below.
     Let $f:D\to \R$ be a $\lambda_1(D)$ eigenfunction, where $D$ satisfies the assumption of the proposition. We know that $f$ is strictly positive on $D$.
     
    The strategy is to use Proposition \ref{prop:isoperimetricinq} to estimate the right-hand side of the Polya-Szeg\"o inequality in Lemma \ref{lemma:pz}.   Hence, we need to ensure that the condition of Proposition \eqref{prop:isoperimetricinq} is met for the super-level sets $\{f>0\}$. For this reason, we restrict the range of $t$. We consider
    \[\tilde t=\sup\{t: |\{f>t\}|\ge 2\sqrt{\vtheta_1}|D|\}\]
    and define \[
    \tilde f=(f-\tilde t)^+,\qquad \hat f=\min\{f,\tilde t\}=:f\wedge\tilde t
    .\]
    Note that $|\{f>\tilde t\}|\le 2\sqrt{\vtheta_1}|D|$ and for any $t\in(0,\tilde t)$, we have $|\{f> t\}|\ge 2\sqrt{\vtheta_1}|D|$.
    We break down the proof into several steps.

    \noindent{\textbf{Step 1.}} We start by giving a lower bound for the Rayleigh quotient of $\tilde f$.
    From 
    \cite[Corollary 5.2]{DFV}, we know that there exists a constant $c_1=c_1(\Omega)$ such that for any $u\in W^{1,2}(\Omega)$ with $|\supp(u)|\le c_1$ we have 
    \[\frac{\int_\Omega|\nabla u|^2}{\int_\Omega u^2}\ge \frac{c_1}{|\supp(u)|^{2/n}}.\]
    We assume $\vtheta_1$ is chosen such that \begin{equation}\label{eq:delta_1condition 1}
    2\sqrt{\vtheta_1}|D|<c_1.\end{equation} Then for $\tilde f=(f-\tilde t)^+$ we have 
    \begin{equation}\label{eq:tildef}
        \frac{\int_\Omega|\nabla \tilde f|^2}{\int_\Omega \tilde f^2}\ge \frac{c_1}{( 2\sqrt{\vtheta_1}|D|)^{2/n}}.
    \end{equation}

    \noindent\textbf{Step 2.} We now estimate the Dirichlet energy of $\hat f$ {using the P\'olya-Szeg\"o inequality}. Up to scaling, we can assume $\int \hat f^2=1$. 
    
    Note that for any $t<\tilde t$, we have $|\{f>t\}|\ge 2\sqrt{\vtheta_1}|D|$. 
Hence,
  \begin{equation}\label{eq:ratio bound}
      \frac{|\{f>t\}\cap \Omega_\delta|}{|\{f>t\}|}\le \frac{|D\cap \Omega_\delta|}{2\sqrt{\vtheta_1}|D|}\le \frac{\vtheta_1|D|}{2\sqrt{\vtheta_1}|D|}=\sqrt{\vtheta_1}/2. \end{equation}
  We can assume $\vtheta_1$ is chosen so that it satisfies  \begin{equation}\label{eq:delta1-condition2}
      \sqrt{\vtheta_1}<\beta,
  \end{equation} where $\beta$ is as in Proposition \ref{prop:isoperimetricinq} with $\epsilon=\epsilon/2$. By assumption we also have $|\{f>t\}|\le|D|\le \vtheta_0$. We can assume that   \begin{equation}\label{eq:delta0-condition1}
      \vtheta_0\le \min\{\alpha,c(\Omega)\}
      \end{equation}
      where $\alpha$ is as in Proposition \ref{prop:isoperimetricinq}, again with $\epsilon=\epsilon/2$, and  $c=c(\Omega)$ is the constant in Lemma~\ref{lemma:pz}.
      Hence, by Proposition \ref{prop:isoperimetricinq}, for a.e. $t<\tilde t$,  we have 
  \begin{equation*}
\per(\{f>t\},\Omega)=|\partial \{f>t\} \cap \Omega| \geq (1-\frac{\epsilon}{2}) \left(1 + C_1\tilde A(\{f>t\})^2 \right) n\omega_n^{1/n} |\{f>t\}|^{\frac{n-1}{n}}.
\end{equation*}
 We can rewrite the above inequality as  
\begin{equation}\label{eq:iso}
\per(\{f>t\},\Omega)^2-(1-\epsilon)\per(\{f^*>t\})^2\geq C_2 (1-\epsilon)  \tilde A(\{f>t\})^2  \mu(t)^{\frac{2(n-1)}{n}},
\end{equation}
where $C_2=2C_1$. We now use Lemma \ref{lemma:pz}, which applies since $\vtheta_0\le c$, {identity \eqref{eq:pz2inq},} and inequality \eqref{eq:iso} to  obtain 

 \begin{equation}\label{eq: pzepsilon}
     \int_{\{f\le \tilde t\}}|\nabla  f|^2-(1-\epsilon) \int_{\{f^*\le \tilde t\}}|\nabla  f^*|^2\ge C_2(1-\epsilon)\int_{0}^{\tilde t}\frac{\tilde A(\{f>t\})^2\mu(t)^{\frac{2(n-1)}{n}}}{|\mu'(t)|}\, dt.   \end{equation}
 Here, $f^*:D^*\to \R$ is the monotone rearrangement of $f$, where $D^*$ is the Euclidean ball centred at the origin with the same volume as $D$.\\

 \noindent\textbf{Step 3.} Let us define $\eta(D):=\frac{\lambda_1(D)}{\lambda_1^D(D^*)}-1$. In this step, we relate the left-hand side of \eqref{eq: pzepsilon} to $\eta(D)$.
Observe that 
\begin{align*}
\int_{\{f^*\le \tilde t\}}|\nabla  f^*|^2&=\int_{D^*}|\nabla(f^*\wedge \tilde t)|^2\\
&\ge \lambda_1^D(D^*)\int_{D^*}(f^*\wedge \tilde t)^2\\
&=\lambda_1^D(D^*)\int_{D}\hat f^2\\
&=\lambda_1^D(D^*).
\end{align*}

On the other hand, we have 
\begin{align*}
    \lambda_1(D)&=\frac{\int_{D}|\nabla f|^2}{\int_D f^2}\\
    &\ge\frac{\int_{D}|\nabla \hat f|^2+\int_{D}|\nabla \tilde f|^2}{1+\int_D\tilde f^2+2\left(\int_D\tilde f^2\right)^{\frac{1}{2}}}\\
\text{\small by inequality \eqref{eq:tildef}~~}    &\ge \frac{\int_{D}|\nabla \hat f|^2+c_1(2\sqrt{\vtheta_1}|D|)^{-\frac{2}{n}}\int_{D}\tilde f^2}{1+\int_D\tilde f^2+2\left(\int_D\tilde f^2\right)^{\frac{1}{2}}}\\
  &\ge\frac{\int_{D}|\nabla \hat f|^2}{1+c_1^{-1}(2\sqrt{\vtheta_1}|D|)^{\frac{2}{n}}{\int_{D}|\nabla \hat f|^2}},
\end{align*}
where in the last inequality, we minimize over possible values of $\int_D\tilde f^2$ as in \cite[Proof of Theorem 5.3]{DFV}. By rearranging, we get
\begin{align*}
\left(1-c_1^{-1}(2\sqrt{\vtheta_1}|D|)^{\frac{2}{n}}\lambda_1(D)\right)\int_{D}|\nabla \hat f|^2\le\lambda_1(D).
\end{align*}
By the assumption, $\lambda_1(D)|D|^{2/n}\le C_0$. Hence, we can assume $\vtheta_1$ is small enough such that 
\begin{equation}\label{eq:delta_1-condition3}
C_0c_1^{-1}(2\sqrt{\vtheta_1})^{\frac{2}{n}}\le\epsilon.\end{equation} Therefore,
\[\int_{D}|\nabla \hat f|^2\le\frac{\lambda_1(D)}{1-\epsilon}.\]
In summary, we get
\begin{equation}\label{eq:main}
     C_2(1-2\epsilon)\int_{0}^{\tilde t}\frac{\tilde A(\{f>t\})^2\mu(t)^{\frac{2(n-1)}{n}}}{|\mu'(t)|}\le \lambda_1^D(D^*)\left(\eta(D)+2\epsilon\right).
     \end{equation}
     Note that we use the trivial inequality $(1-\epsilon)^2\ge 1-2\epsilon$.\\
     
  The remaining part of the proof follows the argument in~\cite{FMP09} for the proof of a quantitative version of the Faber-Krahn inequality. The goal is to obtain a lower bound for $\eta(D)$ in terms of $\tilde A(D)$.\\
  
  \noindent\textbf{Step 4.} {In this step, we establish} a relation between $\tilde A(D)$ and $\tilde A(\{f>t\})$ as in \cite{FMP09}. More precisely, we show that there is a positive constant $C_3=C_3(n)$ such that for any $t\in (0,\tilde t)$ we have
\begin{equation}\label{eq:FMP 1.6}
     \tilde A(D)\le C_3(t\sqrt{|D|}+\eta(D)+\tilde A(\{f>t\})+\epsilon).
 \end{equation}
{\ }\\

 Since $\tilde A(D)\le 2$, inequality \eqref{eq:FMP 1.6} holds for $t>\frac{1}{4}|D|^{-1/2}$ and for any constant $C_3\ge 8$, see~\cite{FMP09}. Thus we assume $t\le\frac{1}{4}{|D|}^{-1/2}$.

 Let $B\subset\mathbb R^n$ be a ball centered at the origin with $|B|=|D\cap U|$. {W.l.o.g., we can also assume $|D\setminus U|\le |D\cap U|$. Thus 
 \begin{equation}\label{eq:volD}
 |B|\le |D|\le 2|B|.
 \end{equation}} For any $x_0\in \R^n$, following the same line of argument as on \cite[Page 59-60]{FMP09}, we have
  \begin{align*}
     |B|A(D\cap U)&\le|(D\cap U)\Delta(x_0+B)|=2|(x_0+B)\setminus(D\cap U)|\\
     &\le 2|(x_0+B)\setminus(\{f>t\}\cap U)|\\
     &\le 2|(x_0+B\cap\{f^*\le s\})\setminus(\{f>t\}\cap U)|+2|(x_0+\{f^*>s\})\setminus(\{f>t\}\cap U)|\\
     &\le 2\left(|B\cap\{f^*\le s\}|+|(x_0+\{f^*>s\})\setminus(\{f>t\}\cap U)|\right)\\
     &\le 2\left(|B\cap\{f^*\le s\}|+|(x_0+\{f^*>s\})\Delta(\{f>t\}\cap U)|\right)
 \end{align*}
 where {$s\ge t$} is chosen such that $|\{f^*>s\}|=|\{f>t\}\cap U|$.
 Thus, optimizing over  $x_0$ and then taking the infimum over all such $U$, we get
 \[ \tilde A(D)\le2\left(\frac{| B\cap\{f^*\le s\}|}{|B|}+ \tilde A(\{f>t\})\right).\]

Let us note that, to define $s$, we have implicitly relied on the assumption that the distribution function $\mu$ is continuous. This assumption is satisfied here. Indeed, for any $\tau>0$, we have
\[\lim_{\sigma\to\tau^-}\mu(\sigma)-\mu(\tau)=|\{f\ge\tau\}|-|\{f>\tau\}|=|\{f=\tau\}|,\]
so $\mu$ is continuous if and only if all the level sets of $f$ have measure zero. Since $f$ solves the elliptic partial differential equation $\left(\Delta-\lambda_1(D)\right)f=0$, it is real-analytic in $D$. Furthermore, since $\lambda_1(D)>0$ and $f>0$ in $D$, $f$ is nowhere constant in $D$. We conclude using the known result that the level sets of a non-constant real-analytic function in a connected open set have measure zero (see for instance \cite{Mit2020ZeroSet} for a short self-contained proof). Let us also note that while the analyticity of $f$ would not hold (and would in fact not be meaningful) if $\Omega$ was a domain in an $n$-dimensional Riemannian manifold rather than in $\R^n$,  the level sets would still have measure zero. This can be proved, for example, by following the approach in \cite{Mit2020ZeroSet}, using the unique continuation property for elliptic PDEs instead of analyticity.
 
We first establish an upper bound for $s$ in terms of $t$.\\

\noindent\textit{Claim.} There exists a positive constant $c_2=c_2(\Omega,C_0,\epsilon)$ such that for  \begin{equation}\label{eq:delta1-condition4}
     \vtheta_1\le c_2,
 \end{equation} we have $$\sqrt{|D|}(s-t)\le \epsilon.$$

\begin{proof}[Proof of the Claim]We first show that  $\mu(s)\in\left(\left(1-\frac{\sqrt{\vtheta_1}}{2}\right)\mu(t), \mu(t)\right) $. Indeed,
 \begin{align*}
  \mu(t)\ge \mu(s)=  |\{f^*> s\}|&=|\{f>t\}\cap U|\\
    &\ge |\{f>t\}\cap \Omega_\delta^c|\\
    &= |\{f>t\}|-|\{f>t\}\cap \Omega_\delta|\\
 \text{\small by \eqref{eq:ratio bound}} \hspace{3mm}  &\ge \left(1-\frac{\sqrt{\vtheta_1}}{2}\right)|\{f>t\}|\\
 &= \left(1-\frac{\sqrt{\vtheta_1}}{2}\right)\mu(t).
 \end{align*}
We note that
\begin{equation}\label{eq:mainestimate}
(s-t)\essinf_{\tau\in(t,s)}(-\mu'(\tau))\le-\int_t^s \mu'(\tau)\,d\tau = \mu(t) - \mu(s).
\end{equation}
Hence, it is enough to find a lower bound for  $\essinf_{\tau\in(t,s)}(-\mu'(\tau))$. The following is true for a.e. $\tau\in (t,s)$.
\[
-\mu'(\tau) = \int_{\{f=\tau\}} \frac{d\cH^{n-1}}{|\nabla f|}
\geq \frac{(\mathcal{H}^{n-1}(\{f=\tau\}))^2}{\int_{\{f=\tau\}} |\nabla f|},
\]
where the last inequality follows from the Cauchy–Schwarz inequality.
Note that $\mu(\tau)\le \theta_0$ and 
\begin{equation*}
      \frac{|\{f>\tau\}\cap \Omega_\delta|}{|\{f>\tau\}|}\le \frac{|D\cap \Omega_\delta|}{\tfrac{1}{2}|\{f>t\}|}\le \frac{\vtheta_1|D|}{\sqrt{\vtheta_1}|D|}=\sqrt{\vtheta_1}. \end{equation*}
Hence, we can apply Proposition \ref{prop:isoperimetricinq} or \cite[Theorem 4.1]{DFV} to get 
\[
-\mu'(\tau)  \geq C\frac{\mu(\tau)^{\frac{2(n-1)}{n}}}{\int_{\{f=\tau\}} |\nabla f|}.
\]
Here and in the following calculation $C$  denotes a constant depending only on the dimension, and it may vary from one inequality to the next. \\
Now, using the fact  that  $f$ is the $\lambda_1(D)$-eigenfunction, we get
\begin{align*}
\int_{\{f=\tau\}} |\nabla f| &=\int_{\{f=\tau\}} \partial_n f =  \int_{\{f>\tau\}} \Delta f \\&= \int_{\{f>\tau\}} \lambda_1(D) f \\&\leq  \lambda_1(D) |D|^{1/2} \|f\|_{L^2(D)}\\&\leq\lambda_1(D) |D|^{1/2} (1+\|\tilde f\|_{L^2(D)}).
\end{align*}
 We now estimate $\|\tilde f\|_{L^2(D)}$. Using inequality \eqref{eq:tildef}, we obtain
   \begin{align*}
       \|\tilde f\|_{L^2(D)}&\le c_1^{-1/2}(2\sqrt{\vtheta_1}|D|)^{1/n}\left(\int_D|\nabla \tilde f|^2\right)^{1/2}\\
       &\le c_1^{-1/2}(2\sqrt{\vtheta_1}|D|)^{1/n}(\int_D|\nabla  f|^2)^{1/2}\\
       &\le c_1^{-1/2}(2\sqrt{\vtheta_1})^{1/n}\left(|D|^{\frac{2}{n}}\lambda_1(D)\right)^{1/2}{\left(1+\|\tilde f\|_{L^2(D)}\right)}\\
       &\le c_1^{-1/2}(2\sqrt{\vtheta_1})^{1/n}C_0^{1/2}{\left(1+\|\tilde f\|_{L^2(D)}\right)}
   \end{align*}
where in the last inequality we use the assumption   $\lambda_1(D)|D|^{2/n}\le C_0$. {Let $$A:= c_1^{-1/2}(2\sqrt{\vtheta_1})^{1/n}C_0^{1/2}.$$
  Thus, for $\vtheta_1$ small enough such that
  \begin{equation}\label{eq:delta1-condition5}1-A>0,\qquad\text{and}\qquad
  \frac{A}{1-A}  <\epsilon,
 \end{equation} we have 
  \begin{equation}\label{eq:tildefupperbound}
      \|\tilde f\|_{L^2(D)}\le \frac{A}{1-A}\le \epsilon.
      \end{equation}}
Therefore,
\[
-\mu'(\tau) \geq C \frac{\mu(\tau)^{\frac{2(n-1)}{n}}}{\lambda_1(D) |D|^{1/2}} \geq C \frac{\mu(t)^{\frac{2(n-1)}{n}}}{\lambda_1(D) |D|^{1/2}}. 
\]
Using the condition
   $\mu(t) - \mu(s) \leq \frac{\sqrt{\theta_1}}{2} \mu(t)$ 
and replacing the above estimate in \eqref{eq:mainestimate}, we get:
\begin{equation}
 C \frac{\mu(t)^{\frac{2(n-1)}{n}}}{\lambda_1(D) |D|^{1/2}} (s-t)\le \frac{\sqrt{\theta_1}}{2} \mu(t). 
\end{equation}
Rearranging and using $\mu(t) \geq 2\sqrt{\theta_1}|D|$, we conclude:
\begin{align*}
s-t &\leq C \frac{\sqrt{\theta_1} \lambda_1(D) |D|^{1/2}}{ \mu(t)^{\frac{2(n-1)}{n}-1}}\le C\frac{\sqrt{\theta_1} \lambda_1(D) |D|^{1/2}}{(\sqrt{\theta_1}|D|)^{\frac{n-2}{n}}}\\&=C\theta_1^{\frac{1}{n}} (\lambda_1(D) |D|^{2/n}) |D|^{-\frac{1}{2}}\le \frac{CC_0\theta_1^{\frac{1}{n}} }{\sqrt{|D|}}.
\end{align*}
\end{proof}
 
  Thus, {to obtain \eqref{eq:FMP 1.6},} it reduces to show that 
 \begin{equation}\label{eq:bf*}
 \frac{|B\cap\{f^*\le s\}|}{|B|}\le C_4(s\sqrt{|D|}+\eta(D)+\epsilon)\le C_4(t\sqrt{|D|}+\eta(D)+2\epsilon)\end{equation}
 for every $s\in [t,t+\epsilon)$ and $t\le \min\{\tilde t, {\frac{1}{4}}|D|^{-1/2}\}$. We can assume \label{page} $\epsilon\le \frac{1}{4}|\Omega|^{-1/2}\le\frac 14|D|^{-1/2}$. Thus \begin{equation}\label{sbound}
  s\le\frac{1}{2}|D|^{-1/2}.   
 \end{equation}
 
We consider the function 
 $f_s(x)=(f^*(x)-s)^+\in H^1_0((1-r)B)$, where  $r>0$ is chosen such that $\{f_s>0\}=\{f^*>s\}=(1-r)B$. 
 Note that 
 \begin{equation}\label{eq:nr|B|}
     |B\cap\{f^*\le s\}|=|B\setminus \{f_s>0\}|=(1-(1-r)^n)|B|\le nr|B|.\end{equation}
 Hence, we need to estimate $r$.
We can use 
 $f_s(x)$ as a test function for $\lambda_1^D((1-r)B)$. \begin{align*}
     \frac{\lambda_1^D(B)}{(1-r)^2}=\lambda_1^D((1-r)B)&\le \frac{\int_{(1-r)B}|\nabla f_s|^2}{\int_{(1-r)B}f_s^2}
     \le \frac{\int_{B}|\nabla f^*|^2}{\int_{(1-r)B}f_s^2}\\
    \text{\small By Lemma \ref{lemma:pz}~~~~} &\le \frac{\int_{D}|\nabla f|^2}{\int_{(1-r)B}f_s^2}\\
    &\le\frac{\lambda_1^D(D^*)(1+\eta(D))\int_{D}f^2}{\int_{(1-r)B}f_s^2}\\
     &\le\frac{\lambda_1^D(B)(1+\eta(D))\int_{D}f^2}{\int_{(1-r)B}f_s^2}
 \end{align*}
Taking the square root of both sides and using the fact that $\|f\|_{L^2(D)}\le 1+\|\tilde f\|_{L^2(D)}\le 1+\epsilon$, we get:
\begin{align}\label{eq:L2 norm tildef}
     \frac{1}{(1-r)}
    &\le\frac{(1+\eta(D))^{1/2}(1+\epsilon}){(\int_{(1-r)B}f_s^2)^{1/2}} .
    \end{align}
 
  After rearranging inequality \eqref{eq:L2 norm tildef} , we get 
 \begin{align}\label{eq:upperbound}
   \nonumber  \left(\int_{(1-r)B}f_s^2\right)^{1/2}&\le(1-r)(1+\eta(D))^{1/2}\left(1+\epsilon\right)\\
  \nonumber   &\le (1-r)(1+\frac{\eta(D)}{2})(1+\epsilon)\\
 \nonumber      &\le (1+\epsilon)(1-r+\frac{\eta(D)}{2})\\
     &\le (1+\epsilon)(1-r)+\eta(D).
\end{align}

 On the other hand,  we have  
\begin{align}\label{eq:sqrt>0}
     \nonumber \left(\int_{(1-r)B}f_s^2\right)^{1/2}&= \left(\int_{(1-r)B}(f^*-s)_+^2\right)^{1/2}\\
 \nonumber     &= \left(\int_{(1-r)B}(f^*-s)^2\right)^{1/2}\\
 \nonumber     \text{\small by the triangle inequality\hspace{2mm}}&\ge \left(\int_B(f^*)^2-\int_{B\setminus(1-r)B} (f^*)^2\right)^{1/2}-s((1-r)^n|B|)^{1/2}\\
 \nonumber    &\ge \left(\int_B(f^*)^2-\int_{B\setminus(1-r)B} (f^*)^2\right)^{1/2}-s|B|^{1/2}\\
    \nonumber  &\ge \left(\int_{D^*}(f^*)^2-\left(\int_{D^*\setminus B}+\int_{B\setminus(1-r)B}\right) (f^*)^2\right)^{1/2}-s|B|^{1/2}\\
 \text{\small using $\|f\|_{L^2(D)}\ge\|\hat f\|_{L^2(D)}=1$}\hspace{3mm}  \nonumber   &\ge \left(1-\left(\int_{D^*\setminus B}+\int_{B\setminus(1-r)B}\right) (f^*)^2\right)^{1/2}-s|B|^{1/2}\\
    \nonumber   &\ge \left(1-s^2\left(|D^*\setminus (1-r)B)|\right)\right)^{1/2}-s|B|^{1/2}\\
     &\ge \left(1-s^2|D^*|\right)^{1/2}-s|B|^{1/2}\\
    \nonumber   &\ge 1-s|D^*|^{1/2}-s|B|^{1/2}\\
   \nonumber    &\ge 1-s\left(\frac{|B|}{1-\vtheta_1}\right)^{1/2}-s|B|^{1/2}.
 \end{align}
The inequality, \eqref{eq:volD} and the bound on $s$ in \eqref{sbound} imply that $1 - s^2|D^*| \ge \frac{1}{2} > 0$, so \eqref{eq:sqrt>0} and the two inequalities above are valid.
 The last inequality is the consequence of the following.  
 We use the assumption $|D\cap \Omega_\delta|\le \vtheta_1|D|$ to get 
$
 |D\cap U^c|\le |D\cap \Omega_\delta|\le \vtheta_1|D|.
 $
Thus, $$|B|=|D\cap U|=|D|-|D\cap U^c|\ge|D|(1-\vtheta_1).$$\\
{Assuming \begin{equation}\label{cond:delta1}
    1-\vtheta_1\ge \frac{1}{4},
\end{equation}} we have
\begin{equation}\label{eq:lowerbound}
    \left(\int_{(1-r)B}f_s^2\right)^{1/2}\ge 1-3s|B|^{1/2}.
\end{equation}
Combining and rearranging \eqref{eq:upperbound} and \eqref{eq:lowerbound}, we get
\[r\le \epsilon+\eta(D)+3s\sqrt{|B|}\le \epsilon+\eta(D)+3s\sqrt{|D|}.\]
Putting this into \eqref{eq:nr|B|}, we obtain inequality \eqref{eq:FMP 1.6} with $C_3=\max\{8, 6n\}$.\\

\noindent\textbf{Step 5.}
For the final step, we use \eqref{eq:main} to show that there exists $t\in (0,\tilde t)$ such that $\tilde  A(\{f>t\})$ is bounded above in terms of $\eta(D)$. Then we use it in inequality \eqref{eq:FMP 1.6} to conclude. 

Using the identity $\lambda_1^D(D^*)|D^*|^{\frac 2n}=\lambda_1^D(\bb)|\bb|^{\frac 2n}$, we can rewrite inequality \eqref{eq:main}  as 
 \begin{equation}\label{eq:main2}
     \int_{0}^{\tilde t}\frac{\tilde A(\{f>t\})^2\mu(t)^{\frac{2(n-1)}{n}}}{|\mu'(t)|}\le C_5|D|^{-\frac{2}{n}}\left(\eta(D)+2\epsilon\right).
     \end{equation}

     W.l.o.g, we assume $\eta(D)\le \delta_3$ where $\delta_3<1$ is small enough {(depending only on $n$)} to be chosen later. We now use \eqref{eq:bf*} to obtain a lower bound for $\mu(t)$. Note that in \eqref{eq:bf*}, $|B|=|D\cap U|$ where $U\subseteq\Omega$ is any open set containing $\Omega\setminus\Omega_\delta$. We take  $U=\Omega$, and thus  $|B|=|D\cap \Omega|=|D|$. Hence, for every $t\le \min\{{\frac{1}{4}}|D|^{-1/2},\tilde t\}$, we have
\begin{align*}
     \mu(t)&=|\{f>t\}|\\
    &\ge|D|-|D\cap \{f\le t\}|\\
     &\ge |B|-|B\cap\{f^*\le t\}|\\
     &\ge |B|-|B|C_4(t\sqrt{|D|}+\eta(D)+\epsilon).
\end{align*}
Let $t_1=\min\left\{\frac{1}{6C_4\sqrt{|D|}},{\frac{1}{4\sqrt{|D|}}},\tilde t\right\}$, $\delta_3\le \frac{1}{6C_4}$ and $\epsilon\le \frac{1}{6C_4}$. Then for any $t\in (0,t_1)$, using \eqref{eq:volD},
$$\mu(t)\ge\frac{|B|}{2}=\frac{|D|}{2}$$ for any $t\in (0,t_1)$. As a result we obtain
 \begin{equation}\label{eq:main3}
     \int_{0}^{t_1}\frac{|D|^{2}\tilde A(\{f>t\})^2}{|\mu'(t)|}\le {C_6}\left(\eta(D)+2\epsilon\right).
     \end{equation}

    We now follow the same line of argument as in \cite{FMP09} and introduce two positive parameters $\theta$ and $\sigma$ (that are fixed below). Let us estimate the $1$-dimensional Hausdorff measure of the following two sets. 
     \[I_1:=\left\{t\in(0,t_1):\frac{\tilde A(\{f>t\})^2}{|\mu'(t)|}\ge \theta\right\},\qquad I_2:=\left\{t\in(0,t_1):|\mu'(t)|\ge |D|\theta^{-\sigma}\right\}.\]
     From inequality \eqref{eq:main3}, we have $$|I_1|\le \frac{C_6(\eta(D)+2\epsilon)}{|D|^{2}\theta},\qquad |I_2|\le \int_{I_2}|\mu'(t)||D|^{-1}\theta^{\sigma}\le \theta^{\sigma},$$
     and therefore
     \[|I_1\cup I_2|\le \left(\frac{C_6(\eta(D)+2\epsilon)}{|D|^2\theta}+\theta^{\sigma}\right).\]
     Taking $\theta=\left(\frac{\eta(D)+2\epsilon}{|D|^2}\right)^{\frac{1}{1+\sigma}}$ and $\sigma=\frac{1}{3}$, we get:
 \[|I_1\cup I_2|\le C_7\left(\frac{\eta(D)+2\epsilon}{|D|^2}\right)^{\frac{\sigma}{1+\sigma}} =\frac{C_7(\eta(D)+2\epsilon)^{\frac 14}}{\sqrt{|D|}}.\]
Now note that 
\begin{align*}
\tilde t^2 |D|&\ge\int_{\{f\le \tilde t\}}f^2+\int_{\{f> \tilde t\}}\tilde t^2= \int_{D}\hat f^2=1.
\end{align*}
Thus, $\tilde t\ge\frac{1}{\sqrt{|D|}}$. As a result $t_1\ge \frac{C_8}{\sqrt{|D|}}$, where $C_8=\min\{\frac{1}{6C_4},\frac 14\}$.
Assuming $\delta_3$ and $\epsilon$ are sufficiently small so that $\epsilon \le\frac{\delta_3}{2}$ and $2C_7(2\delta_3)^{1/4}\le C_8$ implies that the set
 \begin{align*}
     \mathcal B&=(0,{t_1})\cap\left(0,\frac{2C_7(\eta(D)+2\epsilon)^{1/4}}{\sqrt{D}}\right)\setminus (I_1\cup I_2)\\&=\left(0,\frac{2C_7(\eta(D)+2\epsilon)^{1/4}}{\sqrt{D}}\right)\setminus (I_1\cup I_2).\end{align*}
 has nonzero measure. In particular for any $t\in \mathcal B$, we have 
\[\tilde A(\{f>t\})^2\le |\mu'(t)|\theta\le|D| \theta^{1-\sigma}=|D|\left(\frac{\eta(D)+2\epsilon}{|D|^2}\right)^{\frac{1-\sigma}{1+\sigma}}=(\eta(D)+2\epsilon)^{\frac 12}.\]
Plugging into \eqref{eq:FMP 1.6} for such $t$, we get
    \begin{align}\label{eq:sigma=1/3}
        \tilde A(D)&\le C\left((\eta(D)+2\epsilon)^{\frac{\sigma}{1+\sigma}}+\eta(D)+\epsilon+(\eta(D)+2\epsilon)^{\frac{1-\sigma}{2(1+\sigma)}}\right)\\
        \nonumber&\le C(\eta(D)+2\epsilon)^{\frac{1}{4}}.
        \end{align}
       Note that the choice of $\sigma=\frac{1}{3}$ is also made so that the power of the first and the last term in the right-hand side of \eqref{eq:sigma=1/3} are equal. We conclude that
        \[\lambda_1(D)\ge \left(1-2\epsilon+C\tilde A(D)^4\right)\lambda_1^D(D^*).\]

        \noindent Note that we only get a quantitative version when the term $C\tilde A(D)^4-2\epsilon$ is positive.
\end{proof}

\end{document}